\documentclass[12pt]{imsart}
\usepackage{amsthm,amsmath}
\usepackage[colorlinks,citecolor=blue,urlcolor=blue]{hyperref}

\usepackage{amsmath,amsfonts,amssymb,amsthm}

\usepackage{enumitem}

\RequirePackage{amsthm,amsmath,amsfonts,amssymb}
\RequirePackage[numbers]{natbib}
\RequirePackage[colorlinks,citecolor=blue,urlcolor=blue]{hyperref}
\RequirePackage{graphicx}
\usepackage{lineno,hyperref}

\startlocaldefs
\modulolinenumbers[5]

\usepackage[utf8]{inputenc}
\usepackage[english]{babel}

\usepackage[margin=1in]{geometry}
\usepackage{bm}
\usepackage{graphicx}
\usepackage{amssymb,amsmath,amsthm,mathrsfs}
\usepackage{epstopdf}
\usepackage{algorithm}
\usepackage{amsfonts,delarray}
\usepackage{algorithm,algcompatible,amsmath}
\usepackage{rotating,color}
\usepackage{subfigure}
\usepackage{multirow}
\usepackage{titlesec,textcase}
\usepackage{longtable}
\usepackage{bbm}
\usepackage{rotating}

\newtheorem{Theorem}{Theorem}[section]

\newtheorem{Corollary}[Theorem]{Corollary}
\newtheorem{Proposition}[Theorem]{Proposition}

\newtheorem{Definition}[Theorem]{Definition}

\theoremstyle{definition}

\RequirePackage[normalem]{ulem} %DIF PREAMBLE
\definecolor{rp}{RGB}{83,54,106}

\def\boxit#1{\vbox{\hrule\hbox{\vrule\kern6pt\vbox{\kern6pt#1\kern6pt}\kern6pt\vrule}\hrule}}

\endlocaldefs

\allowdisplaybreaks

\begin{document}
\begin{frontmatter}
\title{The interplay between network transitivity and community structure}

\runtitle{ transitivity
and community structure  }
%\thankstext{T1}{A sample additional note to the title.}
\runauthor{ }
\begin{aug}

\author[A]{\fnms{Mingao} \snm{Yuan}\ead[label=e1]{myuan2@utep.edu}},
\author[B]{\fnms{Irin} \snm{Rahman}\ead[label=e2]{irin.rahman@ndsu.edu}},
\author[C]{\fnms{Chengay S} \snm{Wangchuk}\ead[label=e3]{cswangchuk@miners.utep.edu}} \and
\author[D]{\fnms{Minglian} \snm{Lin}\ead[label=e4]{mlin2@utep.edu}}

%%%%%%%%%%%%%%%%%%%%%%%%%%%%%%%%%%%%%%%%%%%%%%
%% Addresses                                %%
%%%%%%%%%%%%%%%%%%%%%%%%%%%%%%%%%%%%%%%%%%%%%%

\address[A]{Department of Mathematical Sciences,
The University of Texas at El Paso,
\printead{e1}}

\address[B]{Department of Statistics,
North Dakota State University,
\printead{e2}}

\address[C]{Department of Mathematical Sciences,
The University of Texas at El Paso,
\printead{e3}}

\address[D]{Department of Mathematical Sciences,
The University of Texas at El Paso,
\printead{e4}}

\end{aug}

\begin{abstract}
Recent empirical observations suggest that network transitivity is highly correlated with community structure in many real-world networks. In this paper, we theoretically investigate this relationship by deriving the limits of the global and average clustering coefficients for the geometric block model (GBM). Both limits exhibit a phase transition; specifically, the functional forms of the limit functions differ between the weak and strong community structure strength regimes. For a GBM with balanced communities, the limits of the global and average clustering coefficients are identical, whereas these limits differ for unbalanced communities. In general, the clustering coefficients do not exhibit a monotonic relationship with community structure  strength. Particularly, for a balanced GBM where the within-community edge probability is a constant multiple of the between-community edge probability, the limit decreases from  $3/4$ to $3/5$ and subsequently increases toward an asymptotic upper bound of $3/4$ as the multiple grows from one. A similar pattern is observed for the global clustering coefficient in unbalanced settings, where both limits exhibit an explicit dependence on community size.

\end{abstract}



\end{frontmatter}

\section{Introduction}
\label{S:1}
In a network or graph, community structure refers to the phenomenon where nodes are organized into clusters characterized by dense internal connections and relatively sparse ties to the rest of the network. These cohesive subsets reveal the system's underlying organization and functional building blocks. Detecting these communities is essential for understanding the dynamics of complex systems \cite{A17,CY06,N03,F10}.

Network transitivity refers to the phenomenon where the neighbors of a node tend to be connected to one another. The global and average clustering coefficients are standard metrics for quantifying network transitivity \cite{WS98, N09}. The global clustering coefficient is the ratio of three times the number of triangles to the number of connected triplets (2-paths), representing the extent to which nodes sharing a common neighbor tend to be connected. Conversely, the local clustering coefficient measures the fraction of a node's neighbors that are interconnected. The average clustering coefficient is simply the average of these local values across all nodes. These metrics are widely applied in network data analysis \cite{BBA01,CLX01,LCYC14}.

Transitivity and community structure are frequently co-occurring features in real-world networks \cite{LKSF10}. While the relationship between them may appear intuitive—suggesting that higher transitivity fosters stronger communities and vice versa—several counterexamples exist. For instance, a fully connected network possesses maximal transitivity but lacks distinct community structure. Conversely, in a network composed of fully connected bipartite communities, the community structure may be pronounced while the transitivity remains zero. To clarify this relationship, recent studies \cite{OLC13,PVS21} have experimentally investigated how these two properties mutually influence one another. Empirical results suggest that  transitivity generally increases with community strength, and the scale of these communities also plays a significant role \cite{OLC13,PVS21}.

This paper presents a theoretical investigation into the relationship between transitivity and community structure. We derive the limits for both global and average clustering coefficients within a Geometric Block Model (GBM), a framework designed to represent community structures in networks with spatial geometry and dependence  structure \cite{GMPS23}. Our results show that both limits undergo a phase transition, with distinct functional forms emerging in weak versus strong community structure strength regimes. The  limits of the clustering coefficients are identical for GBM with balanced communities, and they diverge in unbalanced settings. Notably, the clustering coefficients do not generally exhibit a monotonic relationship with the strength of the community structure. Specifically, in a balanced GBM where the within-community edge probability 
is a constant multiple of the between-community probability, the limit 
decreases from $3/4$ to $3/5$ as the multiple increases from one to four; while 
it subsequently recovers as the multiple grows from four, it remains  
below $3/4$ (see Corollary \ref{basper} and Figure  \ref{gcp}).
 It is worth noting that when the multiple equals one, the GBM reduces to a standard random geometric graph without community structure, which possesses asymptotic clustering coefficients of 
$3/4$. These findings suggest that the introduction of community structure into a random geometric graph generally suppresses the clustering coefficients; moreover, these coefficients change non-monotonically as the community structure strength increases.
 In unbalanced GBM, the global clustering coefficient follows a similar pattern, and both limits depend explicitly on the size of the communities.

The remainder of this paper is organized as follows. In Section 2, we formally introduce the clustering coefficients and the Geometric Block Model. Section 3 presents our main results regarding the global clustering coefficient, while Section 4 focuses on the average clustering coefficient. Finally, Section 5 provides the mathematical proofs for our findings.

\medskip

Notations: Let \(c_{1}\) and \(c_{2}\) be two positive constants. For two positive sequences \(\{a_n\}\) and \(\{b_n\}\), we write \(a_n = \Theta(b_n)\) if \(c_1 \leq a_n/b_n \leq c_2\) for sufficiently large \(n\). We denote \(a_n = O(b_n)\) if \(a_n/b_n \leq c_2\), and \(a_n = o(b_n)\) (or \(b_n = \omega(a_n)\)) if \(\lim_{n \to \infty} a_n/b_n = 0\). For a sequence of random variables \(\{X_n\}\), \(X_n = O_P(a_n)\) indicates that \(X_n/a_n\) is bounded in probability, while \(X_n = o_P(a_n)\) signifies that \(X_n/a_n\) converges to zero in probability. We use \(\mathbb{I}[E]\) to denote the indicator function of an event \(E\). Furthermore, \(\mathbb{E}[X_n]\) and \(\mathbb{P}[E]\) denote the expectation and probability, respectively. For a finite set \(V\), \(\vert{}V\vert{}\) denotes its cardinality. $\mathbb{Z}$ denotes the set of all integers.

\section{ Preliminaries }\label{main}

A graph (or network) $\mathcal{G} = (\mathcal{V}, \mathcal{E})$ consists of a vertex set $\mathcal{V} = \{1, \dots, n\}$ and an edge set $\mathcal{E}$. For an \textit{undirected graph}, $\mathcal{E}$ is a collection of subsets of $\mathcal{V}$ such that $|e|=2$ for all $e \in \mathcal{E}$. The \textit{adjacency matrix} $A \in \{0, 1\}^{n \times n}$ is defined by $A_{ij} = 1$ if $\{i, j\} \in \mathcal{E}$ and $A_{ij} = 0$ otherwise; by symmetry, $A_{ij} = A_{ji}$ and $A_{ii} = 0$.

The \textit{degree} of node $i$ is given by $d_i = \sum_j A_{ij}$. A \textit{triangle} is a set of three mutually adjacent vertices, while a \textit{2-path} consists of three distinct vertices connected by exactly two edges. For example, nodes $i, j, k$ form a triangle if $A_{ij}A_{jk}A_{ki} = 1$, and a 2-path if $A_{ij}A_{jk} = 1$. A graph is \textit{random} if the upper-triangular entries $\{A_{ij}\}_{i < j}$ are random variables.

\begin{Definition}\label{def0}
Given a graph $\mathcal{G}=(\mathcal{V},\mathcal{E})$  with adjacency matrix $A$,  the \textbf{global clustering coefficient}  is defined as
\[\mathcal{C}_n=\frac{\sum_{i\neq j\neq k}A_{ij}A_{jk}A_{ki}}{\sum_{i\neq j\neq k}A_{ij}A_{jk}}.\]
The  \textbf{average clustering coefficient} is defined as
\begin{equation}\label{clustercoe}
\overline{\mathcal{C}}_n=\frac{1}{n}\sum_{i=1}^n\frac{\sum_{j\neq k} A_{ij}A_{jk}A_{ki}}{d_i(d_i-1)},
\end{equation}
where any summation terms with $d_i=0,1$ are set to be zero.
\end{Definition}

The global clustering coefficient $\mathcal{C}_n$ is defined as the ratio of the number of triangles to the number of 2-paths \cite{WS98, N09}. 
It quantifies the transitivity of the graph, representing the tendency of two neighbors of a common vertex to be connected. The average clustering coefficient $\overline{\mathcal{C}}_n$ is an alternative measure of global clustering or transitivity \cite{WS98, N09}. The $i$-th summation term in $\overline{\mathcal{C}}_n$ is called the local clustering coefficient of node $i$. The average clustering coefficient is the average of the local clustering coefficients of all vertices.

The \textit{Geometric Block Model} is a latent space model designed to capture the inherent geometric features and clustering properties of networks with community structure \cite{GMPS23}. In this study, we focus on the following formulation:

\begin{Definition}\label{def1}
Let $n$ be a positive  integer and  $\tau \in [0, 1]$ be a constant such that $\tau n \in \mathbb{Z}$. Suppose $r_s, r_d \in [0, 0.5]$ are real numbers satisfying $r_s \geq r_d$. Let $V_1, V_2 \subset \{1, 2, \dots, n\}$ be disjoint subsets such that $|V_1| = \tau n$ and $|V_2| = (1-\tau) n$.
 Given independent uniform random variables $X_1, X_2, \dots, X_n$ on $[0, 1]$, the \textbf{Geometric Block Model} (GBM) $\mathcal{G}_n(r_s, r_d)$ is defined by the adjacency matrix $A$, where $A_{ii} = 0$ and for $i \neq j$:
\[
A_{ij} = \mathbb{I} \left[ d(X_i, X_j) \leq r_d + (r_s - r_d) \mathbb{I}[Z_i = Z_j] \right]
\]
where $d(X_i, X_j) = \min \{|X_i - X_j|, 1 - |X_i - X_j|\}$, and the community indicators are $Z_i = 1$ if $i \in V_1$ and $Z_i = 0$ if $i \in V_2$.
\end{Definition}

In \(\mathcal{G}_n(r_s, r_d)\), nodes are uniformly distributed on the unit interval \([0, 1]\), and each node $i$ is assigned a community membership indicator $Z_i$. Two nodes are connected if their distance is less than a threshold that depends on their community membership.
In contrast to the distance metric used in \cite{GMPS23}, we employ a periodic distance on the unit interval to eliminate boundary effects.
The network consists of two communities, $V_1$ and $V_2$, where the parameter $\tau \in [0, 1]$ determines the relative size of each community. In the GBM of \cite{GMPS23}, the parameter is set to \(\tau = 1/2\); that is, the two communities have an equal number of nodes, a configuration we refer to as balanced community structure or balanced communities. When $\tau\neq1/2$, we say the community structure is unbalanced.
 The parameters $r_s$ and $r_d$ govern the edge densities within and across communities, respectively. Specifically, two nodes belonging to the same community are connected if their distance is less than $r_s$, whereas nodes in different communities are connected if their distance is less than $r_d$. The assumption $r_s \geq r_d$ reflects the homophilic property that nodes within the same community are more likely to be connected than those in different communities.  When \(\tau \in \{0, 1\}\) or \(r_s = r_d\), \(\mathcal{G}_n(r_s, r_d)\) simplifies to a standard random geometric graph without community structure.
In Definition \ref{def1}, we assume $\tau n$ is an  integer for convenience; however, the model $\mathcal{G}_n(r_s, r_d)$ can be analogously defined using the nearest integer to $\tau n$ when $\tau n$ is not an integer, and our main results remain valid in that case. Throughout this paper, we allow \(r_{s}\) and \(r_{d}\) to scale with \(n\).

\section{The global clustering coefficient of GBM}

This section provides the asymptotic expression for the global clustering coefficient of $\mathcal{G}_n(r_s, r_d)$. Firstly, we present the expectations of a triangle and a 2-path appearing in $\mathcal{G}_n(r_s, r_d)$. These results are necessary to derive the limits of the clustering coefficients. They are also of independent interest for understanding the number of triangles and 2-paths of $\mathcal{G}_n(r_s, r_d)$.

\begin{Proposition}\label{proptria}
    Let $A \sim \mathcal{G}_n(r_s, r_d)$.  
    \begin{enumerate}
    \item[(I)] Suppose $r_d \leq r_s < 2r_d$. Then for $i\in\{1,2,3\}$, we have
    \begin{equation}\label{trian1}
        \mathbb{E}[A_{12}A_{13}A_{23}|X_i]=3r_s^2\mathbb{I}(Z_1=Z_2=Z_3)+(4r_sr_d-r_s^2)\big(1-\mathbb{I}(Z_1=Z_2=Z_3)\big).
    \end{equation}
    \item[(II)] Suppose $r_s \geq 2r_d$. Then for $i\in\{1,2,3\}$, we have
    \begin{equation}\label{trian2}
               \mathbb{E}[A_{12}A_{13}A_{23}|X_i]=3r_s^2\mathbb{I}(Z_1=Z_2=Z_3)+4r_d^2\big(1-\mathbb{I}(Z_1=Z_2=Z_3)\big).
    \end{equation}
    \item[(III)]  For $i\in\{1,2,3\}$, we have
    \begin{align}\nonumber
               \mathbb{E}[A_{12}A_{13}|X_i]&=4r_s^2\mathbb{I}(Z_1=Z_2=Z_3)+4r_sr_d\big(\mathbb{I}(Z_1=Z_2\neq Z_3)+\mathbb{I}(Z_1=Z_3\neq Z_2)\big)\\ \label{trian3}
               &\quad +4r_d^2\mathbb{I}(Z_1\neq Z_2=Z_3).
    \end{align}
\end{enumerate}
\end{Proposition}

Taking the expectation with respect to $X_i$ in (\ref{trian1}), (\ref{trian2}), and (\ref{trian3}) yields the unconditional expectations of a triangle and a 2-path. These expectations depend on whether the involved nodes belong to the same community.
 Interestingly, the expectation of a triangle undergoes a phase change and has different expressions for $r_d \leq r_s < 2r_d$ and $r_s \geq 2r_d$. These results lead to the phase change of the limits of the clustering coefficients at $r_s=2r_d$.

\begin{Theorem}\label{asygc}
Suppose $A \sim \mathcal{G}_n(r_s, r_d)$ with $r_s = \Theta(r_d)$ and $nr_s = \omega(1)$. The global clustering coefficient $\mathcal{C}_n$ satisfies:
\begin{enumerate}
    \item[(I)] If $r_d \leq r_s < 2r_d$, then
    \begin{equation}\label{gc1}
        \mathcal{C}_n = \frac{[3 - 9\tau(1-\tau)]r_s^2 + 3\tau(1-\tau)(4r_sr_d - r_s^2)}{[4 - 12\tau(1-\tau)]r_s^2 + 8\tau(1-\tau)r_sr_d + 4\tau(1-\tau)r_d^2} + o_P(1).
    \end{equation}
    \item[(II)] If $r_s \geq 2r_d$, then
    \begin{equation}\label{gc2}
        \mathcal{C}_n = \frac{[3 - 9\tau(1-\tau)]r_s^2 + 12\tau(1-\tau)r_d^2}{[4 - 12\tau(1-\tau)]r_s^2 + 8\tau(1-\tau)r_sr_d + 4\tau(1-\tau)r_d^2} + o_P(1).
    \end{equation}
\end{enumerate}
\end{Theorem}

Theorem \ref{asygc} establishes the asymptotic limit of the global clustering coefficient for the Geometric Block Model \(\mathcal{G}_n(r_s, r_d)\).   Notably, these limits exhibit a phase transition at $r_s = 2r_d$, resulting in distinct functional forms for the regimes $r_s < 2r_d$ and $r_s \geq 2r_d$. This discrepancy stems from the differing triangle formation 
probabilities between the two cases, as shown in Proposition \ref{proptria}.
 In the special case where $r_s = r_d$, the model $\mathcal{G}_n(r_s, r_d)$ reduces to the standard random geometric graph without community structure. In this regime, the global clustering coefficient is asymptotically $3/4$, which is consistent with existing results (see \cite{Y25c, YF25} for example). The limits in \eqref{gc1} and \eqref{gc2} are symmetric about $\tau = 1/2$, 
since the term $\tau(1-\tau)$ is symmetric around $\tau = 1/2$. This is 
intuitively expected because one community contains $\tau n$ nodes while 
the other contains $(1-\tau)n$ nodes. Thus, choosing $\tau = t$ or 
$\tau = 1-t$ for any $t \in (0, 0.5)$ results in the same community structure. 

The conditions $r_s = \Theta(r_d)$ and $nr_s = \omega(1)$ in Theorem \ref{asygc} imply that the 
connection densities within and between communities are of the same order 
and that the expected node degree diverges as $n \to \infty$. These 
assumptions are standard in the community detection literature \cite{GMPS23}.

 The general expressions in \eqref{gc1} and \eqref{gc2} are relatively complex. The underlying behavior of the global clustering coefficient becomes more 
theoretically transparent in certain special settings: (a) the case of two 
balanced communities where $\tau = 1/2$, and (b) the case where $r_s$ is a 
constant multiple of $r_d$. The model explored in \cite{GMPS23} satisfies 
both of these conditions.

The following corollary presents the limit of the global clustering coefficient 
for $\mathcal{G}_n(r_s, r_d)$ with balanced communities.

\begin{Corollary}\label{gccor1}
    Suppose $A \sim \mathcal{G}_n(r_s, r_d)$ with $r_s = \Theta(r_d)$, $nr_s = \omega(1)$, and $\tau = 1/2$. The global clustering coefficient $\mathcal{C}_n$ satisfies the following asymptotic expressions:
    \begin{itemize}
        \item[(I)] If $r_d \leq r_s < 2r_d$, then
        \begin{equation}\label{gccoreq1}
            \mathcal{C}_n = \frac{3r_s r_d}{(r_s + r_d)^2} + o_P(1).
        \end{equation}
        \item[(II)] If $r_s \geq 2r_d$, then
        \begin{equation}\label{gccoreq2}
            \mathcal{C}_n = \frac{3(r_s^2 + 4r_d^2)}{4(r_s + r_d)^2} + o_P(1).
        \end{equation}
    \end{itemize}
\end{Corollary}

As shown in Corollary \ref{gccor1}, the global clustering coefficient for $\mathcal{G}_n(r_s, r_d)$ with balanced communities admits simple asymptotic expressions. These limits more clearly exhibit  a phase transition at 
$r_s = 2r_d$, resulting in distinct functional forms for the regimes 
$r_s < 2r_d$ and $r_s \geq 2r_d$. For $\mathcal{G}_n(r_s, r_d)$ with an unbalanced community structure, the 
limits established in Theorem \ref{asygc} do not admit expressions as 
simple as those in Corollary \ref{gccor1}.

Theorem \ref{asygc} establishes a theoretical foundation for examining the impact of community structure on the global clustering coefficient. This offers a formal basis for the empirical observations in \cite{PVS21}, which suggest that clustering is highly correlated with community structure. To isolate and analyze this impact more effectively, we examine the special case where \(r_{s}\) is a constant multiple of \(r_{d}\).
This choice of parameters encompasses the model setting investigated in 
\cite{GMPS23}. Specifically, we let $r_s = \lambda r_d$ with $\lambda \geq 1$. 
Under this setting, $\mathcal{G}_n(r_s, r_d)$ exhibits a stronger 
community structure as $\lambda$ increases. This occurs because the 
intra-community edge probability becomes significantly higher than the 
inter-community edge probability; thus, larger values of $\lambda$ 
correspond to more pronounced community partitions.

\begin{Corollary}\label{corunbg}
Suppose $A \sim \mathcal{G}_n(r_s, r_d)$ with $r_s = \lambda r_d$ for a constant $\lambda \geq 1$ and $nr_s = \omega(1)$. The global clustering coefficient $\mathcal{C}_n$ satisfies the following asymptotic expression:
\[
\mathcal{C}_n = g(\lambda, \tau) + o_P(1),
\]
where
\begin{equation}\label{glambdta}
g(\lambda, \tau) =
\begin{cases} 
\dfrac{[3 - 9\tau(1-\tau)]\lambda^2 + 3\tau(1-\tau)(4\lambda - \lambda^2)}{[4 - 12\tau(1-\tau)]\lambda^2 + 8\tau(1-\tau)\lambda + 4\tau(1-\tau)}, & \lambda \in [1, 2), \\[12pt]
\dfrac{[3 - 9\tau(1-\tau)]\lambda^2 + 12\tau(1-\tau)}{[4 - 12\tau(1-\tau)]\lambda^2 + 8\tau(1-\tau)\lambda + 4\tau(1-\tau)}, & \lambda \ge 2.
\end{cases}
\end{equation}
\end{Corollary}

The extreme values and monotonic intervals of \(g(\lambda, \tau)\) are summarized in the following proposition.

\begin{Proposition}\label{propcorunbg}
For fixed $\tau \in (0, 1)$, the function $g(\lambda, \tau)$ defined in (\ref{glambdta}) is decreasing on $[1, \lambda^*]$ and increasing on $[\lambda^*, \infty)$ with respect to $\lambda$, where
\[
\lambda^* = \frac{3}{2} + \frac{1}{2} \sqrt{\frac{9 - 11\tau(1-\tau)}{1 - 3\tau(1-\tau)}}.
\]
Moreover, $g(\lambda, \tau) < \frac{3}{4}$ for $\lambda > 1$ and $\tau \in (0,1)$, 
whereas $g(\lambda, \tau) = \frac{3}{4}$ if and only if $\lambda = 1$ or 
$\tau \in \{0, 1\}$.
\end{Proposition}

Corollary \ref{corunbg} provides the limit of the global clustering coefficient as a function of the community structure strength parameter \(\lambda \), and Proposition \ref{propcorunbg} characterizes the properties of this limit function. When \(\lambda = 1\) or \(\tau \in \{0, 1\}\), the GBM \(\mathcal{G}_n(r_s, r_d)\) reduces to a random geometric graph without community structure. The limit is \(3/4\) in this case; otherwise, it is strictly less than \(3/4\), indicating that the presence of community structure reduces the clustering coefficient.
As the community structure strengthens with increasing \(\lambda \), the limit exhibits a non-monotonic trend, initially decreasing before eventually increasing.

The limit function \eqref{glambdta} simplifies significantly for 
$\mathcal{G}_n(r_s, r_d)$ under a balanced community structure, which 
includes one of the models studied in \cite{GMPS23}.

\begin{Corollary}\label{basper}
    Suppose $A \sim \mathcal{G}_n(r_s, r_d)$ with $r_s = \lambda r_d$ for a constant $\lambda \geq 1$, $nr_s = \omega(1)$, and $\tau = 1/2$. The global clustering coefficient $\mathcal{C}_n$ satisfies the following asymptotic expression:
    \[\mathcal{C}_n =f(\lambda)+o_P(1),\]
    where
\begin{equation}\label{flbabaln}
f(\lambda) =
\begin{cases}
 \dfrac{3\lambda}{(1+\lambda)^2}, & \ \ \lambda \in [1,2), \\[10pt]
 \dfrac{3(4+\lambda^2)}{4(1+\lambda)^2}, & \ \ \lambda \ge 2.
\end{cases}
\end{equation}
Moreover, $f(\lambda)$ decreases on \([1, 4]\), increases on \([4, \infty)\) and $f(\lambda)<f(1)=\frac{3}{4}$ for all $\lambda>1$.
\end{Corollary}

Corollary \ref{basper} provides a simple closed-form expression for the limit of the clustering coefficient as a function of the parameter \(\lambda \) in a GBM with \(\tau = 1/2\) and \(r_s = \lambda r_d\). This specific model—with \(r_d = b\log n / n\)—was previously studied in \cite{GMPS23} in the context of community detection. By Proposition \ref{propcorunbg}, the limit of the global clustering 
coefficient, $f(\lambda)$, is strictly less than $3/4$ for all $\lambda > 1$. 
The function attains its maximum of $3/4$ at $\lambda = 1$ and its minimum 
of $3/5$ at $\lambda =4$. The limit function $f(\lambda)$ decreases on the 
interval $[1, 4]$ and subsequently increases for $\lambda \in [4, \infty)$. 
These behaviors are illustrated in Figure \ref{gcp}.

\begin{figure}[h]
    \centering
    \includegraphics[width=0.5\linewidth]{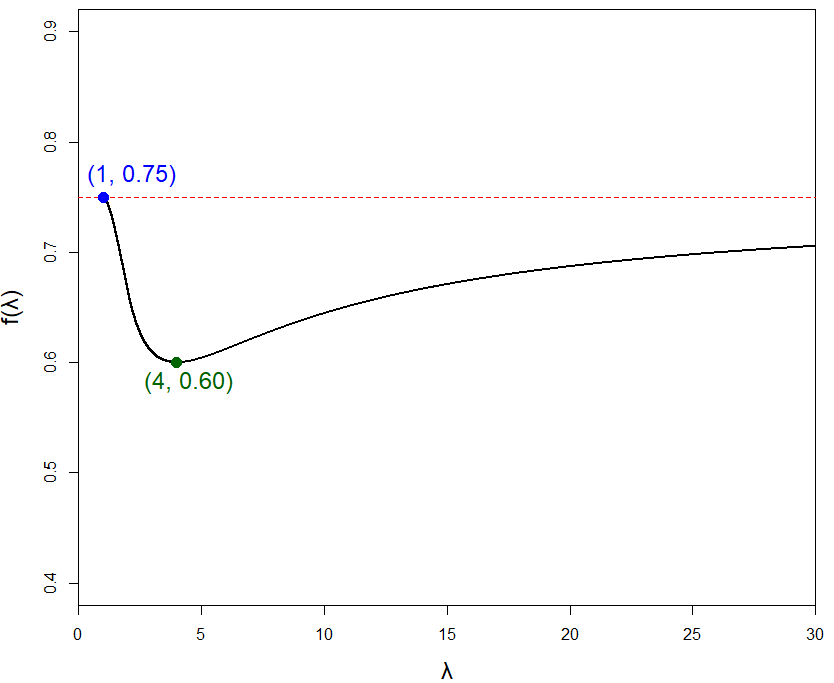}
    \caption{Plot of \(f(\lambda)\) with global maximum at \(\lambda=1\) and global minimum at \(\lambda=4\).}
    \label{gcp}
\end{figure}

\section{The average clustering coefficient of GBM}

This section provides the asymptotic expression for the average clustering coefficient of $\mathcal{G}_n(r_s, r_d)$.

\begin{Theorem}\label{asyalc}
Suppose $A \sim \mathcal{G}_n(r_s, r_d)$ with $r_s = \Theta(r_d)$ and $nr_s = \omega(1)$. The average clustering coefficient $\overline{\mathcal{C}}_n$ satisfies:
\begin{enumerate}
    \item[(I)] If $r_d \leq r_s < 2r_d$, then
    \begin{align} \nonumber
        \overline{\mathcal{C}}_n &= \frac{
      \tau \Bigl[3\tau^2 r_s^2 + (1-\tau^2)(4r_s r_d - r_s^2)\Bigr]
    }{
      4\bigl(\tau r_s + (1-\tau) r_d\bigr)^2
    } 
  + \frac{
      (1-\tau)\Bigl[3(1-\tau)^2 r_s^2 +\tau(2-\tau) (4r_s r_d - r_s^2)\Bigr]
    }{
      4\bigl(\tau r_d+(1-\tau)r_s \bigr)^2
    }\\ 
    &\quad+ o_P(1).
    \end{align}
    \item[(II)] If $r_s \geq 2r_d$, then
    \begin{align} \nonumber
        \overline{\mathcal{C}}_n &= \frac{
      \tau\Bigl[3\tau^2 r_s^2 + 4r_d^2(1-\tau^2)\Bigr]
    }{
      4\bigl(\tau r_s + (1-\tau) r_d\bigr)^2
    }  + \frac{
      (1-\tau)\Bigl[3(1-\tau)^2 r_s^2 + 4r_d^2\tau(2-\tau)\Bigr]
    }{
      4\bigl(\tau r_d+(1-\tau)r_s \bigr)^2
    }+ o_P(1).
    \end{align}
\end{enumerate}
\end{Theorem}

Theorem \ref{asyalc} establishes the asymptotic limit of the average clustering coefficient for the Geometric Block Model \(\mathcal{G}_n(r_s, r_d)\).   Similar to the global case, a phase transition occurs at $r_s = 2r_d$, resulting in distinct functional forms for the regimes $r_s < 2r_d$ and $r_s \geq 2r_d$  due to variations in triangle formation probability.

The general expressions in \eqref{gc1} and \eqref{gc2} are relatively 
complex. To simplify them, we first focus on $\mathcal{G}_n(r_s, r_d)$ 
with two balanced communities ($\tau = 1/2$).

\begin{Corollary}\label{lgccor1}
    Suppose $A \sim \mathcal{G}_n(r_s, r_d)$ with $r_s = \Theta(r_d)$, $nr_s = \omega(1)$, and $\tau = 1/2$. The average clustering coefficient $\overline{\mathcal{C}}_n$ satisfies the following asymptotic expressions:
    \begin{itemize}
        \item[(I)] If $r_d \leq r_s < 2r_d$, then
        \begin{equation}\label{gccoreq1}
            \overline{\mathcal{C}}_n = \frac{3r_s r_d}{(r_s + r_d)^2} + o_P(1).
        \end{equation}
        \item[(II)] If $r_s \geq 2r_d$, then
        \begin{equation}\label{gccoreq2}
            \overline{\mathcal{C}}_n = \frac{3(r_s^2 + 4r_d^2)}{4(r_s + r_d)^2} + o_P(1).
        \end{equation}
    \end{itemize}
\end{Corollary}

Corollary~\ref{lgccor1} shows that for $\mathcal{G}_n(r_s, r_d)$ with balanced 
communities, the average clustering coefficient admits a simple asymptotic 
expression that coincides with the limit of the global clustering coefficient in Corollary \ref{gccor1}.
  In the special case where $r_s = r_d$, the model $\mathcal{G}_n(r_s, r_d)$ 
reduces to a standard random geometric graph. In this case, the average clustering coefficient asymptotically approaches 
$3/4$.

A natural question is whether the average and global clustering coefficients have the same limit  in the unbalanced case. 
To address this, we consider the scenario where $r_s$ is a constant 
multiple of $r_d$.
 
\begin{Corollary}\label{lgccor2}
Suppose $A \sim \mathcal{G}_n(r_s, r_d)$ with $r_s = \lambda r_d$ for a constant $\lambda \geq 1$ and $nr_s = \omega(1)$. The average clustering coefficient $\overline{\mathcal{C}}_n$ satisfies
\[\overline{\mathcal{C}}_n = h(\lambda,\tau) + o_P(1),\]
where, for $1 \leq \lambda < 2$:
\begin{equation}
    h(\lambda,\tau) = \frac{\tau[3\tau^2\lambda^2 + (4\lambda - \lambda^2)(1-\tau^2)]}{4(\tau\lambda + 1 - \tau)^2} + \frac{(1-\tau)[3(1-\tau)^2\lambda^2 + (4\lambda - \lambda^2)\tau(2-\tau)]}{4((1-\tau)\lambda + \tau)^2},
\end{equation}
and for $\lambda \geq 2$:
\begin{equation}
    h(\lambda,\tau) = \frac{\tau[3\tau^2\lambda^2 + 4(1-\tau^2)]}{4(\lambda\tau + 1 - \tau)^2} + \frac{(1-\tau)[3(1-\tau)^2\lambda^2 + 4\tau(2-\tau)]}{4(\tau + (1-\tau)\lambda)^2}.
\end{equation}
\end{Corollary} 
The limit function $h(\lambda, \tau)$ has a complex form, making it difficult 
to explicitly characterize its extreme values and monotonic intervals as was done for the global 
clustering coefficient in Corollary \ref{corunbg}. Clearly, \(h(\lambda, \tau)\) and \(g(\lambda, \tau)\) in \eqref{glambdta}  differ in functional form; however, obtaining a simple analytical expression for this difference is quite cumbersome. To highlight their difference, we plot \(h(\lambda, \tau)\) and \(g(\lambda, \tau)\) as functions of \(\tau \) for several fixed values of $\lambda\in\{1.5,2, 5,10,25, 50\}$ in Figure \ref{fighg}. The black and green curves represent $h(\lambda, \tau)$ and $g(\lambda, \tau)$, respectively, while the horizontal red dotted line represents the value $3/4$.
As illustrated, the two functions generally differ, with $h(\lambda, \tau)$ with fixed $\tau$ alternating between being larger and smaller than $g(\lambda, \tau)$.  Both functions attain a minimum near $\tau = 1/2$. For both small and large values of $\lambda$, 
both functions approach $3/4$.

\begin{figure}[p]
    \centering   \includegraphics[width=0.45\linewidth]{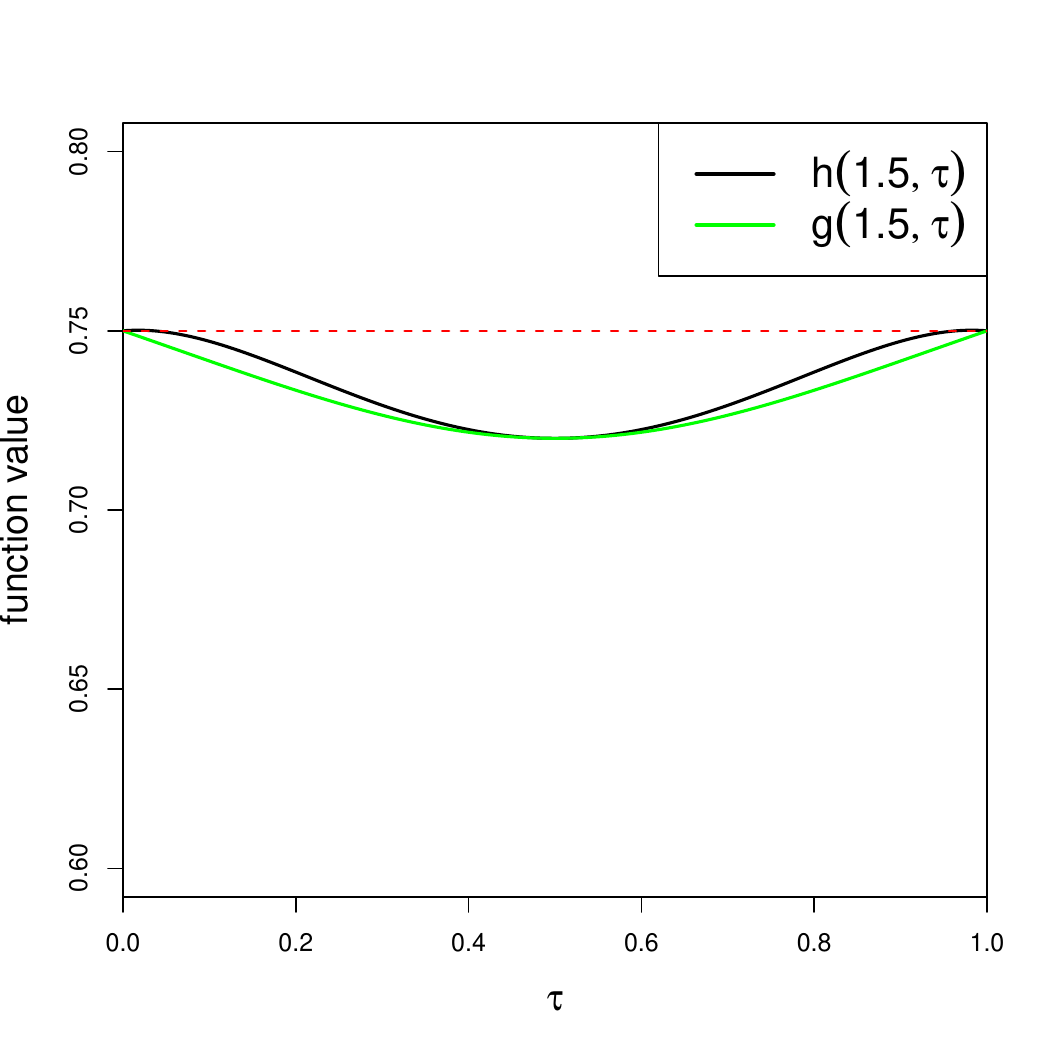}
\includegraphics[width=0.45\linewidth]{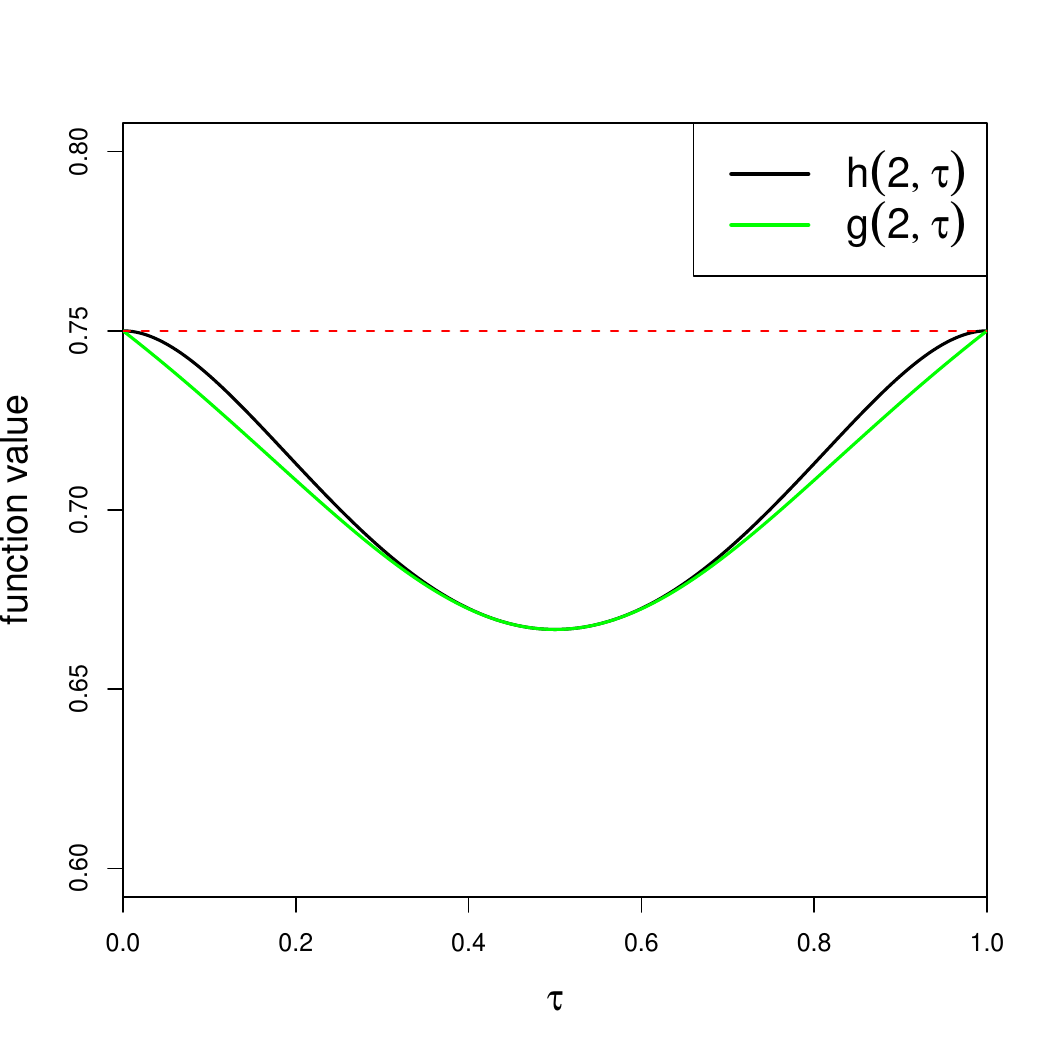}
\includegraphics[width=0.45\linewidth]{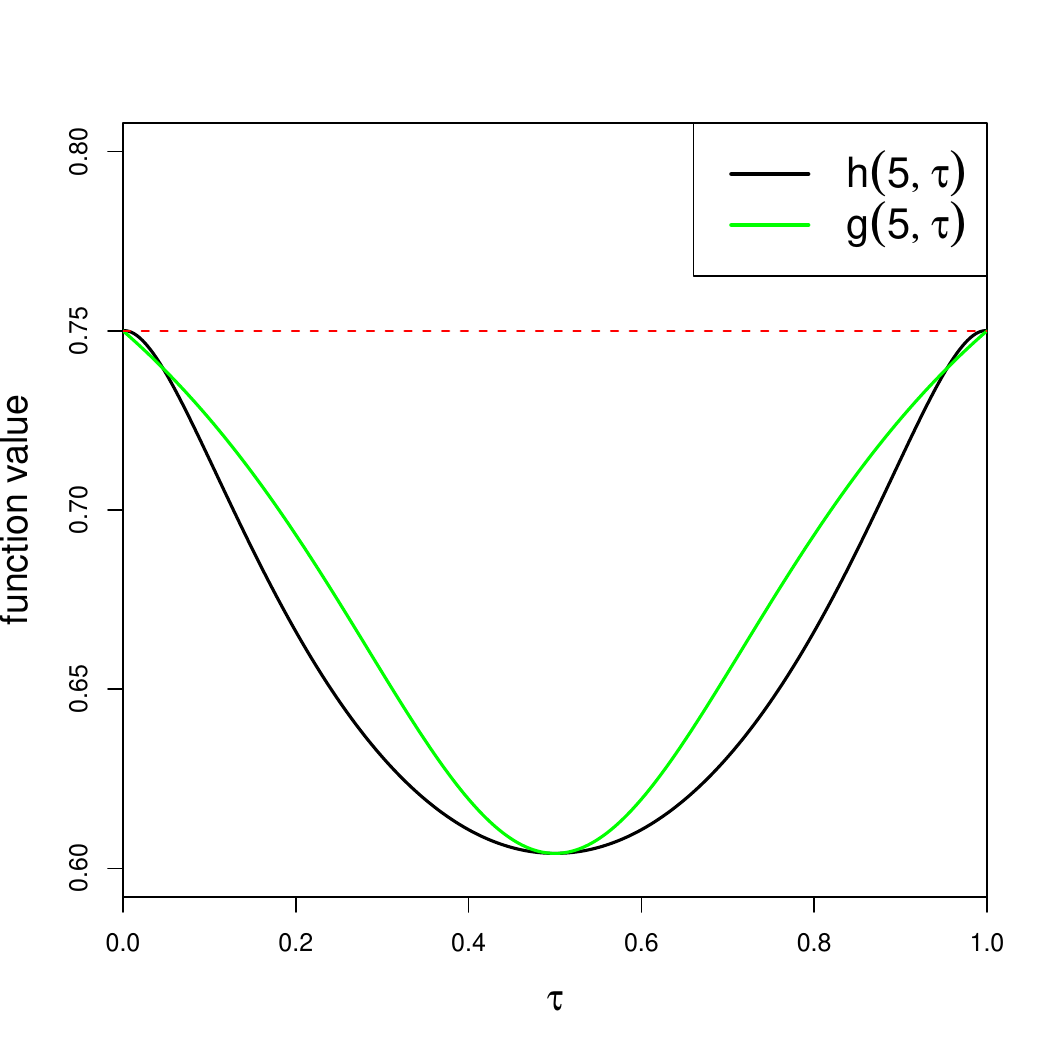}
\includegraphics[width=0.45\linewidth]{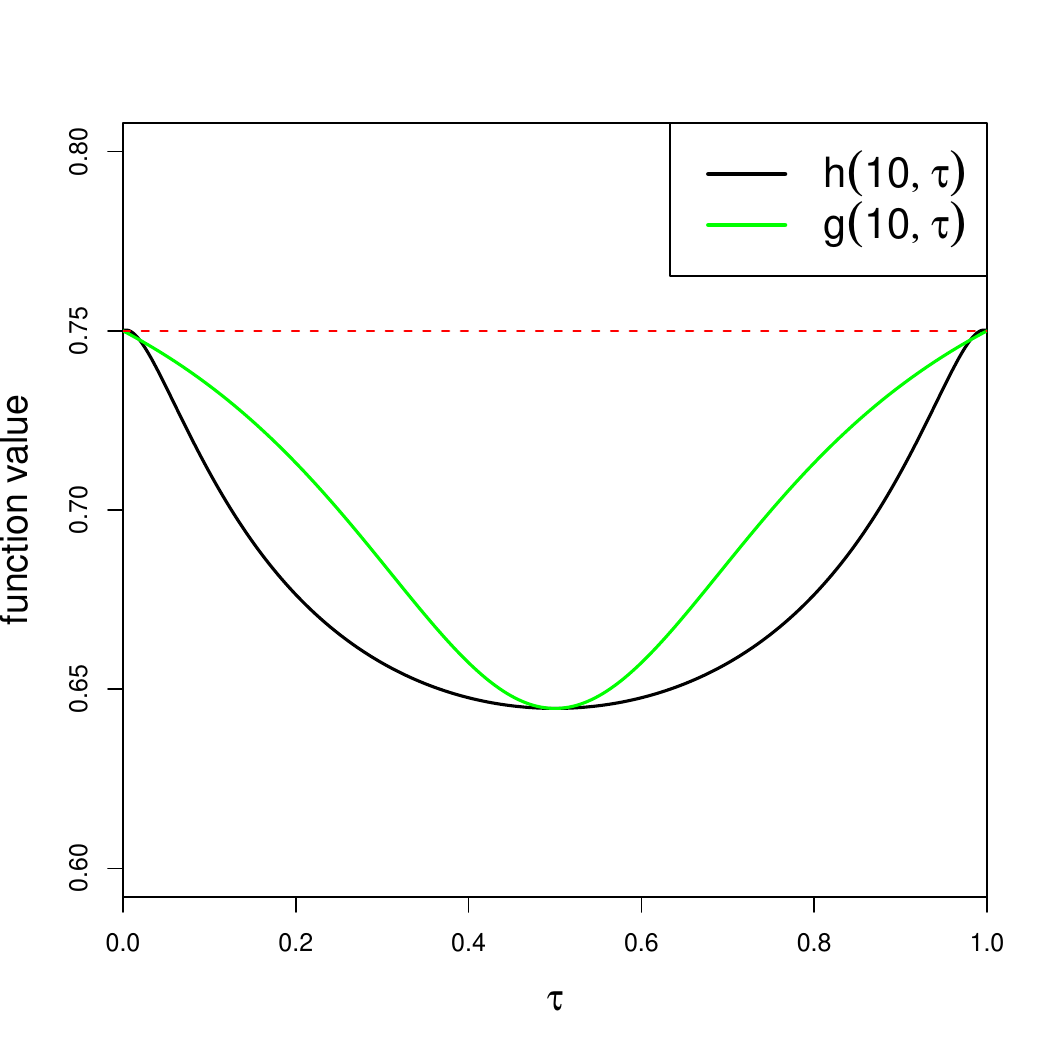}
\includegraphics[width=0.45\linewidth]{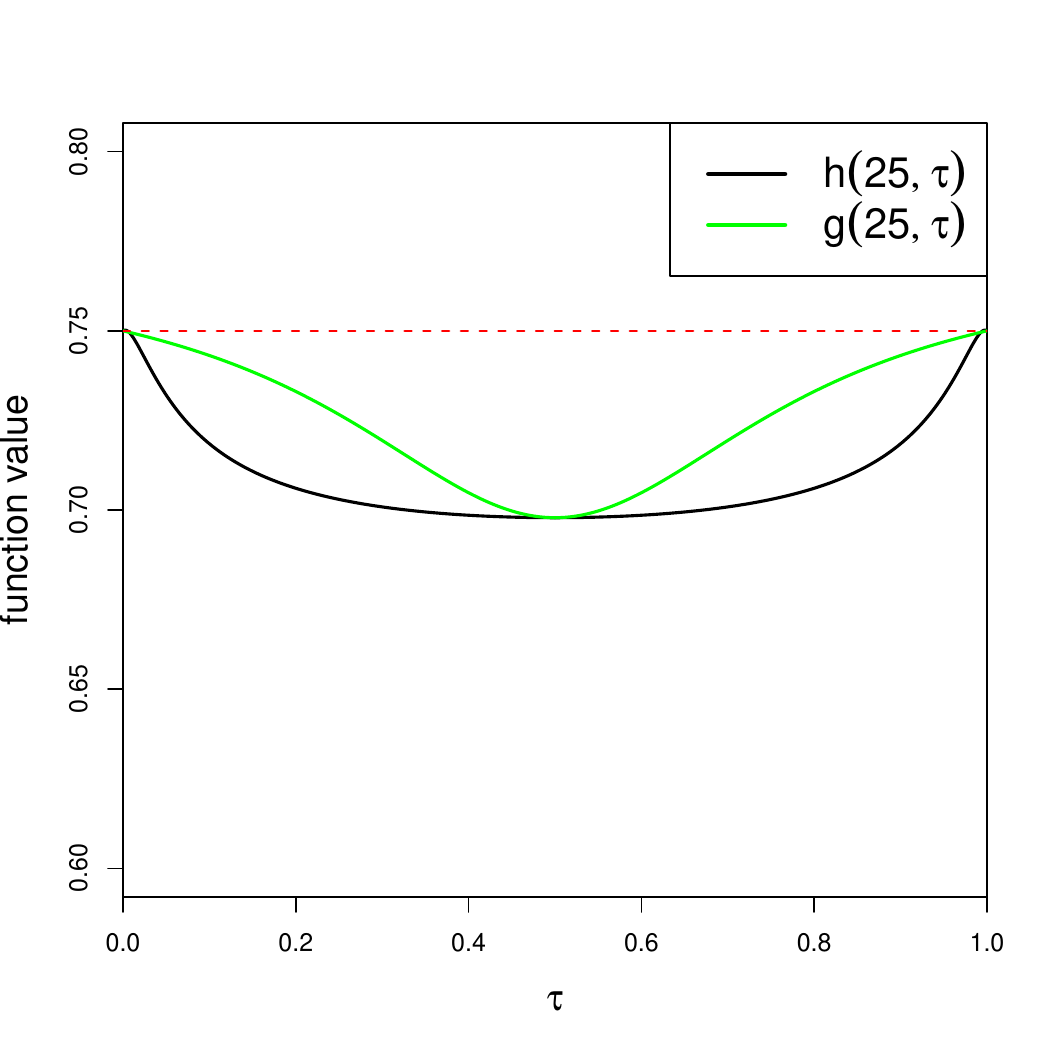}
\includegraphics[width=0.45\linewidth]{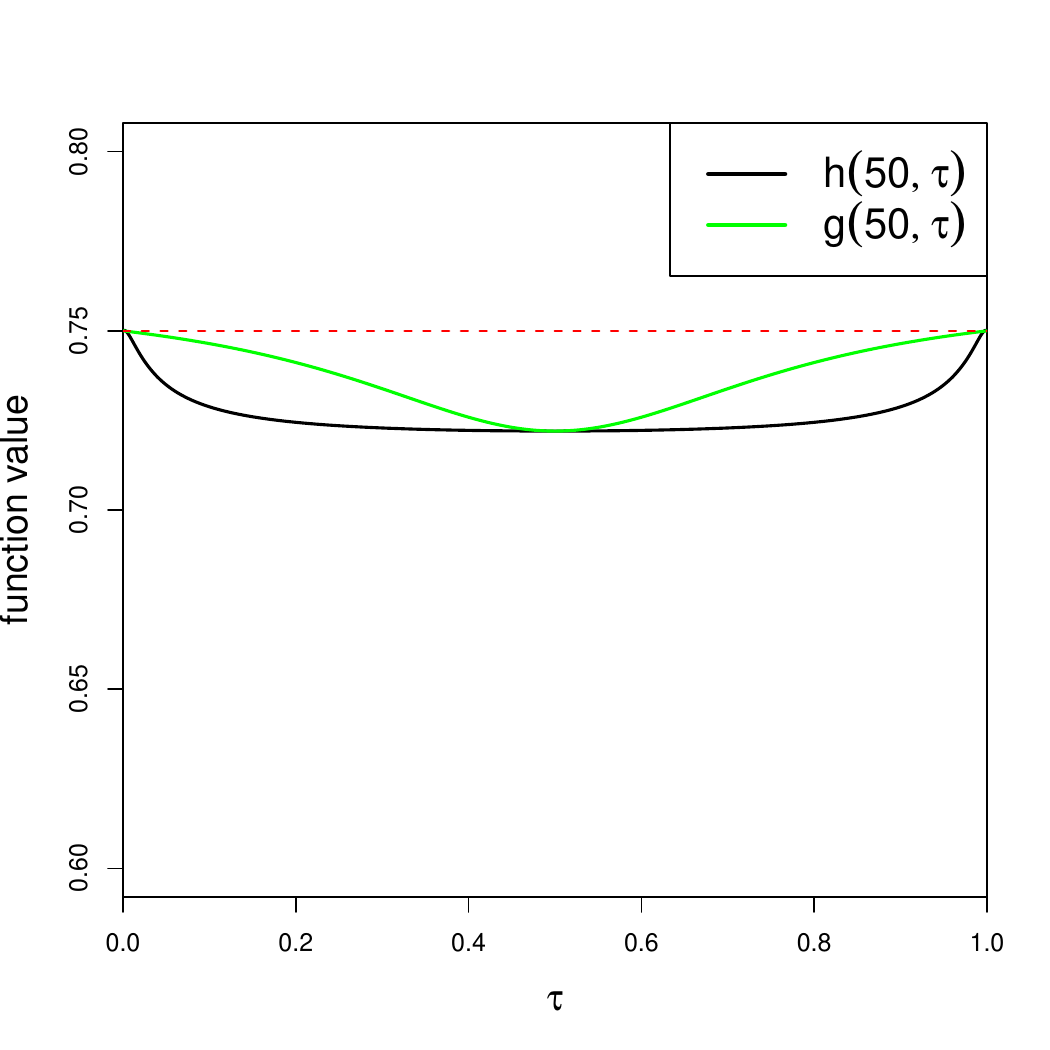}
    \caption{Graphs of $h(\lambda,\tau)$ (black) and $g(\lambda,\tau)$ (green)  for fixed values of $\lambda\in\{1.5,2, 5,10, 25,50\}$.}
    \label{fighg}
\end{figure}

Next, we consider $h(\lambda, \tau)$ as a function of $\lambda$ for a  fixed $\tau$. Interestingly, $h(\lambda, \tau)$ may exceed $3/4$  when $\tau$ is close to $0$ or $1$. Furthermore, the properties of  $h(\lambda, \tau)$ when $\tau$ is close to $0$ or $1$ differ from its  properties when $\tau$ is not close to $0$ or $1$.

	\begin{Proposition}\label{propml1}
		(a). For $ \lambda \in [1, 2)$ and fixed $\tau \in (0, 0.0670) \cup (0.9330, 1)$, the function $h(\lambda,\tau)$ defined in Corollary \ref{lgccor2} has at least one, at most two local maximums greater than $3/4$, occurring at $\lambda^* \in (1, 2)$.\\
        (b). For $\tau = 0.01$, $ h(\lambda, 0.01) $ is increasing on $(1, 1.5377)$, $(3.0224, 8.0742)$, and $(127.6016, \infty)$, and decreasing on $(1.5377, 3.0224)$ and $(8.0742, 127.6016)$.\\
        (c). For $\tau = 0.1$, $ h(\lambda, 0.1)$ is decreasing on $(1, 9.3689)$ and increasing on $(9.3689, \infty)$.
	\end{Proposition}

The function $h(\lambda, 0.01)$ is illustrated in Figure~\ref{0.01},  while the function $h(\lambda, 0.1)$ is shown in Figure~\ref{0.1}.  These figures illustrate the difference between $h(\lambda, \tau)$  and $g(\lambda, \tau)$, thereby highlighting the distinction between 
the average clustering coefficient and the global clustering coefficient.

	\begin{figure}[h]
		\begin{minipage}{\textwidth}
			\centering
			\includegraphics[width=0.75\linewidth]{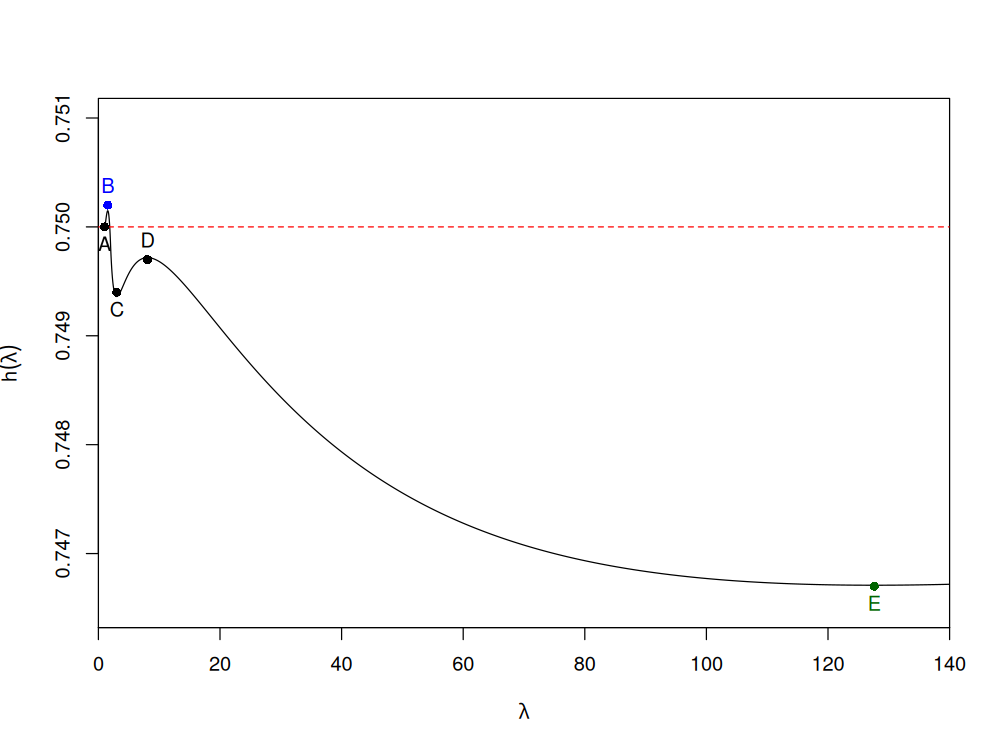}
			\caption{Plot of $h(\lambda, 0.01)$: $A(1,0.75)$, $B(1.5377, 0.7502)$, $C(3.0224,0.7494)$,  $D(8.0742,0.7497)$, and $E(127.6016, 0.7467)$.    }
			\label{0.01}
		\end{minipage}
	\end{figure}

	\begin{figure}[h]
		\begin{minipage}{\textwidth}
			\centering
			\includegraphics[width=0.75\linewidth]{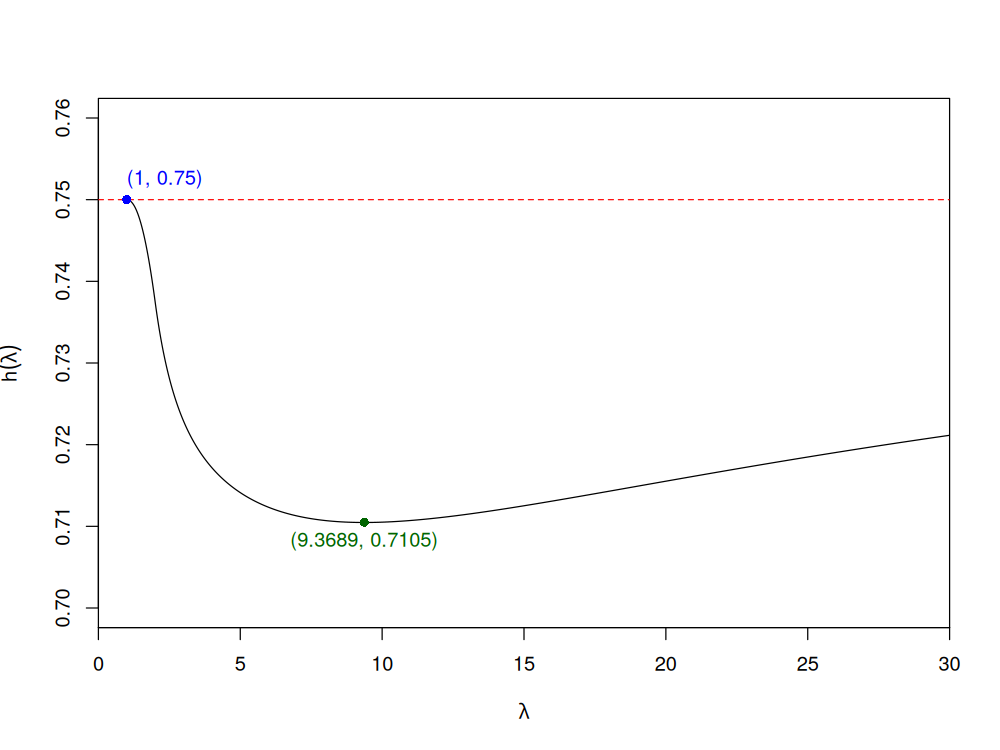}
			\caption{Plot of $h(\lambda, 0.1)$: global maximum at  $\lambda = 1$ and global minimum at $\lambda = 9.3689$.}
			\label{0.1}
		\end{minipage}
	\end{figure}

\section{Proof of main results}

In this section, we provide detailed proofs of the main results.

\subsection{Proof of Proposition \ref{proptria}}

{\bf Proof of (I)}. Suppose $r_d \leq r_s < 2r_d$.   If nodes 1, 2, 3 are in the same community, that is, $Z_1=Z_2=Z_3$, then $\mathbb{E}[A_{12}A_{23}A_{31}|X_i]=3r_s^2$ by Lemma 3.1 in \cite{Y25c}. We only need to consider the case where the three nodes belong to different 
communities. Since there are only two communities, two nodes must share 
the same community while the third is in another. Without loss of 
generality, we assume node 1 and node 2 share a community and node 3 
is in another, that is, $Z_1 = Z_2 \neq Z_3$. 

Note that by symmetry, $\mathbb{E}[A_{12}A_{23}A_{31}|X_1] = \mathbb{E}[A_{12}A_{23}A_{31}|X_2]$; 
thus, we only need to compute $\mathbb{E}[A_{12}A_{23}A_{31}|X_1]$ and 
$\mathbb{E}[A_{12}A_{23}A_{31}|X_3]$. We first evaluate the latter. 
Since the distance $d(X_i, X_j)$ is translation-invariant and the 
$X_i$ are uniformly distributed, $\mathbb{E}[A_{12}A_{23}A_{31}|X_3]$ 
is independent of the specific value of $X_3$. Without loss of 
generality, let $X_3 = x = 0.5$.

Consider the following four scenarios: (a) $X_1<x,X_2<x$; (b) $X_1<x<X_2$; (c) $X_1>x,X_2>x$;(d) $X_2<x<X_1$. Then
\begin{eqnarray}\nonumber
    \mathbb{E}[A_{12}A_{23}A_{31}|X_3=x]&=&\int_{(a)}A_{12}A_{23}A_{31}d_{X_1}d_{X_2}+\int_{(b)}A_{12}A_{23}A_{31}d_{X_1}d_{X_2}\\   \nonumber
    &&+\int_{(c)}A_{12}A_{23}A_{31}d_{X_1}d_{X_2}+\int_{(d)}A_{12}A_{23}A_{31}d_{X_1}d_{X_2}\\ \label{lempreq1} 
    &=&2\int_{(a)}A_{12}A_{23}A_{31}d_{X_1}d_{X_2}+2\int_{(b)}A_{12}A_{23}A_{31}d_{X_1}d_{X_2}.
\end{eqnarray}

 Note that $r_s\geq r_d$. For Case (a), we have
\begin{eqnarray}\label{lempreq2}   \int_{(a)}A_{12}A_{23}A_{31}d_{X_1}d_{X_2}=\int_{x-r_d}^{x}\int_{x-r_d}^{x}d_{X_1}d_{X_2}=r_d^2.
\end{eqnarray}

For Case (b), $A_{12}A_{23}A_{31}=1$ if and only if $x-r_d<X_1<x-(r_s-r_d)$ and $x<X_2<X_1+r_s$  or $x-(r_s-r_d)<X_1<x$ and $x<X_2<x+r_d$. Straightforward calculation yields
\begin{align*}
\int_{x-r_d}^{x-(r_s-r_d)} \int_{x}^{x_1+r_s} dx_2   dx_1
&= \int_{x-r_d}^{x-(r_s-r_d)} (x_1 + r_s - x)  dx_1 \\
&= \left[ \frac{x_1^2}{2} \right]_{x-r_d}^{x-(r_s-r_d)}
+ (r_s - x)(2r_d - r_s) \\
&= \frac{1}{2} \Big[
x^2 - 2x(r_s - r_d) + (r_s - r_d)^2 
- x^2 + 2xr_d - r_d^2
\Big] \\
&\quad + r_s(2r_d - r_s) - x(2r_d - r_s) \\
&= \frac{1}{2} \Big[
(r_s - r_d)^2 - r_d^2
\Big] + r_s(2r_d - r_s) \\
&= \frac{r_s(2r_d - r_s)}{2},
\\
\int_{x-(r_s - r_d)}^{x} \int_{x}^{x + r_d} dx_2   dx_1
&= \int_{x-(r_s - r_d)}^{x} r_d   dx_1 \\
&= r_d (r_s - r_d).
\end{align*}
Then we have
\begin{align}\label{nlempreq3}
\int_{(b)}A_{12}A_{23}A_{31}d_{X_1}d_{X_2}=\frac{r_s(2r_d - r_s)}{2} + r_d(r_s - r_d)= \frac{4r_s r_d - r_s^2 - 2r_d^2}{2}.
\end{align}

Combining (\ref{lempreq1}), (\ref{lempreq2}) and (\ref{nlempreq3})  yields
\begin{eqnarray}\nonumber
    \mathbb{E}[A_{12}A_{23}A_{31}|X_3=x]=2  r_d^2 + 2  \frac{4r_s r_d - r_s^2 - 2r_d^2}{2}=4r_s r_d - r_s^2.
\end{eqnarray}

Now we calculate $\mathbb{E}[A_{12}A_{23}A_{31}|X_1=x]$. Similar to the previous case, we have 
\begin{eqnarray}\label{sbmeq1}
    \mathbb{E}[A_{12}A_{23}A_{31}|X_1=x]&=&2\int_{(a)}A_{12}A_{23}A_{31}d_{X_1}d_{X_2}+2\int_{(b)}A_{12}A_{23}A_{31}d_{X_1}d_{X_2},
\end{eqnarray}
where (a) $X_3<x,X_2<x$ and (b) $X_2<x<X_3$.

Note that $r_s\geq r_d$. For case (b),  we have
\begin{align}\nonumber
\int_{(b)}A_{12}A_{23}A_{31}d_{X_1}d_{X_2}&=\int_{x}^{x+r_d} dx_3 \int_{x_3-r_d}^{x} dx_2\\ \nonumber
&= \int_{x}^{x+r_d} (x + r_d - x_3)  dx_3 \\ \nonumber
&= (x+r_d)r_d - \frac{1}{2}(x^2 + 2xr_d + r_d^2 - x^2) \\ \nonumber
&= r_d^2 - \frac{r_d^2}{2} \\ \label{sbmeq2}
&= \frac{r_d^2}{2}.
\end{align}

For Case (a), $A_{12}A_{23}A_{31}=1$ if and only if $x-r_s<X_2<x-r_d$ and $x-r_d<X_3<X_2+r_d$  or $x-r_d<X_2<x$ and $x-r_d<X_3<x$. Straightforward calculation yields
\begin{align}\nonumber
&\int_{(a)}A_{12}A_{23}A_{31}d_{X_1}d_{X_2}\\ \nonumber
&=\int_{x-r_s}^{x-r_d} dx_2 \int_{x-r_d}^{x_2+r_d} dx_3
+
\int_{x-r_d}^{x} dx_2 \int_{x-r_d}^{x} dx_3\\ \nonumber
&=
\int_{x-r_s}^{x-r_d} (x_2 + 2r_d - x)  dx_2 + r_d^2
\\ \nonumber
&=
\frac{x_2^2}{2}\Big|_{x-r_s}^{x-r_d}
+ (2r_d - x)(r_s - r_d) + r_d^2\\ \nonumber
&= \frac{x^2 - 2r_dx + r_d^2 - x^2 + 2xr_s - r_s^2}{2}
+ 2r_d(r_s - r_d) - x(r_s - r_d) + r_d^2 \\ \nonumber
&= x(r_s - r_d) - \frac{r_s^2 - r_d^2}{2}
+ 2r_d r_s - 2r_d^2 + r_d^2 - x(r_s - r_d) \\ \label{sbmeq3}
&= 2r_s r_d - \frac{r_s^2}{2} - \frac{r_d^2}{2}.
\end{align}

Combining (\ref{sbmeq1}), (\ref{sbmeq2}) and (\ref{sbmeq3})  yields
\begin{eqnarray}\nonumber
    \mathbb{E}[A_{12}A_{23}A_{31}|X_1=x]=4r_s r_d - r_s^2.
\end{eqnarray}
Then the proof of case (I) is complete.

\medskip 

{\bf Proof of (II)}. The case (II) follows from Lemma 3.1 in \cite{Y25c}. 

\medskip 

{\bf Proof of (III)}. Firstly, we consider $\mathbb{E}\left[A_{12}A_{13}\mid X_1\right]$.  By Lemma 3.1 in \cite{Y25c}, $\mathbb{E}  \left[A_{12}\mid X_1\right]=2r_s$ if $Z_1=Z_2$ and $\mathbb{E}  \left[A_{12}\mid X_1\right]=2r_d$ if $Z_1\neq Z_2$.  Note that $X_1,X_2,X_3$ are independent. 
If $Z_1 = Z_2 = Z_3$, then
\begin{align*}
\mathbb{E}\left[A_{12}A_{13}\mid X_1\right]
 =\mathbb{E}  \left[A_{12}\mid X_1\right]
\mathbb{E}  \left[A_{13}\mid X_1\right]
=(2r_s)(2r_s)= 4r_s^2.
\end{align*}
If $Z_1 = Z_2 \ne Z_3$ or $Z_1=Z_3 \ne Z_2$, then
\begin{align*}
\mathbb{E}\left[A_{12}A_{13}\mid X_1\right]
 =\mathbb{E}  \left[A_{12}\mid X_1\right]
\mathbb{E}  \left[A_{13}\mid X_1\right]=(2r_s)(2r_d)
= 4r_s r_d.
\end{align*}
If $Z_2 = Z_3 \ne Z_1$, then
\begin{align*}
\mathbb{E}\left[A_{12}A_{13}\mid X_1\right]
 =\mathbb{E}  \left[A_{12}\mid X_1\right]
\mathbb{E}  \left[A_{13}\mid X_1\right]=(2r_d)(2r_d)
= 4r_d^2.
\end{align*}

Now, we consider $\mathbb{E}\left[A_{12}A_{13}\mid X_2\right]$.
Note that
\begin{align*}
\mathbb{E}\left[A_{12}A_{13}\mid X_2\right]&=\mathbb{E}\left[\mathbb{E}\left[A_{12}A_{13}\mid X_1, X_2\right]\mid X_2\right]
 =\mathbb{E}\left[A_{12}\mathbb{E}\left[A_{13}\mid X_1\right]\mid X_2\right]\\
 &=\big(2r_s\mathbb{I}[Z_1=Z_3]+2r_d\mathbb{I}[Z_1\neq Z_3]\big)\mathbb{E}\left[A_{12}\mid X_2\right]
 .
\end{align*}
If $Z_1 = Z_2 = Z_3$, then 
\begin{align*}
\mathbb{E}\left[A_{12}A_{13}\mid X_2\right]=2r_s\mathbb{I}[Z_1=Z_3]\mathbb{E}\left[A_{12}\mid X_2\right]=2r_s\mathbb{I}[Z_1=Z_3]2r_s=4r_s^2.
\end{align*}
If $Z_1 = Z_2 \ne Z_3$, then
\begin{align*}
\mathbb{E}\left[A_{12}A_{13}\mid X_2\right]=2r_d\mathbb{I}[Z_1\neq Z_3]\mathbb{E}\left[A_{12}\mid X_2\right]=2r_d\mathbb{I}[Z_1\neq Z_3]2r_s
= 4r_s r_d.
\end{align*}
If $Z_1=Z_3 \ne Z_2$, then
\begin{align*}
\mathbb{E}\left[A_{12}A_{13}\mid X_2\right]=2r_s\mathbb{I}[Z_1=Z_3]\mathbb{E}\left[A_{12}\mid X_2\right]=2r_s2r_d
= 4r_s r_d.
\end{align*}
If $Z_2 = Z_3 \ne Z_1$, then  
\begin{align*}
\mathbb{E}\left[A_{12}A_{13}\mid X_1\right]
 =2r_d\mathbb{I}[Z_1\neq Z_3]\mathbb{E}\left[A_{12}\mid X_2\right]=2r_d2r_d
= 4r_d^2.
\end{align*}
Note that $\mathbb{E}\left[A_{12}A_{13}\mid X_2\right]=\mathbb{E}\left[A_{12}A_{13}\mid X_3\right]$.
In summary, the proof of (III) is complete. Then the proof of Proposition \ref{proptria} is complete.

\qed

\subsection{Proof of Theorem \ref{asygc}}
 Without loss of generality, let $V_1=\{1,2,\dots,\tau n\}$ and $V_2=\{\tau n+1,\tau n+2,\dots, n\}$. 
To prove Theorem \ref{asygc}, it suffices to establish the following 
asymptotic relations:
\begin{align}\label{thempreq1}
    \sum_{i \ne j \ne k} A_{ij}A_{ik} &= (1+o_P(1))\sum_{i \ne j \ne k} \mathbb{E}[A_{ij}A_{ik}], \\ 
    \label{thempreq2}
    \sum_{i \ne j \ne k} A_{ij}A_{ik}A_{kj} &= (1+o_P(1))\sum_{i \ne j \ne k} \mathbb{E}[A_{ij}A_{ik}A_{kj}],
\end{align}
and subsequently evaluate the sums of the expectations on the 
right-hand side of \eqref{thempreq1} and \eqref{thempreq2}.

First, we evaluate the leading-order term of the sum of expectations on the 
right-hand side of \eqref{thempreq1} and establish the convergence in 
\eqref{thempreq1}.
Recall that $\tau \in [0,1]$ is a fixed constant and $r_s = \Theta(r_d)$ 
by assumption. By Proposition \ref{proptria}, the sum of the expectations 
in \eqref{thempreq1} is evaluated as follows:
\begin{align}\nonumber
\sum_{i \ne j \ne k} \mathbb{E}[A_{ij}A_{ik}]
&= \Bigl[(\tau n)(\tau n-1)(\tau n-2) + (1-\tau)n((1-\tau)n-1)((1-\tau)n-2)\Bigr] 4r_s^2 \\ \nonumber
  &\quad + 2\Bigl[(\tau n)(\tau n-1)(1-\tau)n + \tau n(1-\tau)n((1-\tau)n-1)\Bigr] 4r_s r_d \\ \nonumber
  &\quad +\Bigl[(\tau n)(\tau n-1)(1-\tau)n + \tau n(1-\tau)n((1-\tau)n-1)\Bigr] 4r_d^2 \\ \nonumber
&= \Bigl[\tau^3 n^3+ (1-\tau)^3 n^3\Bigr] 4r_s^2 + 2\Bigl(\tau^2 n^2 (1-\tau)n + \tau n(1-\tau)^2 n^2\Bigr) 4r_s r_d \\  \nonumber
  &\quad + \Bigl[\tau^2 n^2(1-\tau)n + \tau n(1-\tau)^2 n^2\Bigr] 4r_d^2 +O(n^2r_s^2)\\ \label{exp2pathe}
  &= n^3\Big[\bigl(1 -3\tau(1-\tau)\bigr) 4r_s^2
    + 8 \tau(1-\tau)  r_s r_d
    + 4 \tau(1-\tau)  r_d^2\Big] +O(n^2r_s^2).
\end{align}

We use Markov's inequality to prove (\ref{thempreq1}). To this end, we evaluate the variance of the left-hand side of (\ref{thempreq1}).  For convenience, let $P_n
= \sum_{i \ne j \ne k} \bar{P}_{ijk}$, where $\bar{P}_{ijk}
=A_{ij}A_{jk} - \mathbb{E}[A_{ij}A_{jk}]$. Then
\begin{align}\label{secmom}
\mathbb{E}  \left[P_n^2\right]
&= \sum_{i \ne j \ne k} \sum_{i_1 \ne j_1 \ne k_1}
\mathbb{E}  \left[\bar{P}_{ijk} \bar{P}_{i_1 j_1 k_1}\right].
\end{align}
If $\{i,j,k\}$ and $\{i_1,j_1,k_1\}$ are disjoint, then  $\bar{P}_{ijk}$ and $\bar{P}_{i_1 j_1 k_1}$ are independent. Hence, the $\mathbb{E}  \left[\bar{P}_{ijk} \bar{P}_{i_1 j_1 k_1}\right]=0$. 

Next, we show that if $\{i,j,k\}$ and $\{i_1,j_1,k_1\}$ have exactly 
one vertex in common, the expectation $\mathbb{E} \left[\bar{P}_{ijk} 
\bar{P}_{i_1 j_1 k_1}\right]$ is zero.
 To this end, we only need to show  $\mathbb{E}  \left[\bar{P}_{123} \bar{P}_{145}\right]=0$,  $\mathbb{E}  \left[\bar{P}_{123} \bar{P}_{415}\right]=0$ and $\mathbb{E}\left[\bar{P}_{123} \bar{P}_{425}\right]=0$.  Note that
\begin{align*}
\mathbb{E}  \left[\bar{P}_{123} \bar{P}_{145}\right]
&= \mathbb{E}  \left[
\mathbb{E}  \left[\bar{P}_{123}\mid X_1\right]
\mathbb{E}  \left[\bar{P}_{145}\mid X_1\right]
\right].
\end{align*}
By Proposition \ref{proptria}, $\mathbb{E}  \left[\bar{P}_{123}\mid X_1\right]=0$. Then $\mathbb{E}  \left[\bar{P}_{123} \bar{P}_{145}\right]=0$. Similar arguments apply to the remaining cases where the index sets share only one node. Consequently,  $\mathbb{E}[\bar{P}_{ijk} \bar{P}_{i_1 j_1 k_1}] 
\neq 0$ only if $\{i,j,k\}$ and $\{i_1,j_1,k_1\}$ share at least two nodes. By (\ref{secmom}), we have
\begin{align*}
\mathbb{E}  \left[P_n^2\right]
&= \sum_{i \ne j \ne k \ne l}
\Big(
4 \mathbb{E}  \left[\bar{P}_{ijk}\bar{P}_{ijl}\right]
+ 4 \mathbb{E}  \left[\bar{P}_{ijk}\bar{P}_{ilj}\right]
+ 2 \mathbb{E}  \left[\bar{P}_{ikj}\bar{P}_{il j}\right]
+ 4 \mathbb{E}  \left[\bar{P}_{ijk}\bar{P}_{l ij}\right]
\Big) \\
&\quad + \sum_{i \ne j \ne k}
\Big(
2 \mathbb{E}  \left[\bar{P}_{ijk}^2\right]
+ 4 \mathbb{E}  \left[\bar{P}_{ijk}\bar{P}_{ikj}\right]
\Big).
\end{align*}
By a similar argument as in the proof of (III) of Proposition \ref{proptria}, it is easy to verify that $\mathbb{E}  \left[A_{12}A_{13}A_{14}\right]=O(r_s^3)$. Then
\[|\mathbb{E}  \left[\bar{P}_{ijk}\bar{P}_{ijl}\right]|=\mathbb{E}  \left[A_{ij}A_{jk}A_{jl}\right]+O(r_s^4)=O(r_s^3)+O(r_s^4).\]
Similarly, we have
\[|\mathbb{E}  \left[\bar{P}_{ijk}\bar{P}_{ilj}\right]|=\mathbb{E}  \left[A_{ij}A_{jk}A_{il}A_{jl}\right]+O(r_s^4)\leq \mathbb{E}  \left[A_{ij}A_{jk}A_{jl}\right]+O(r_s^4)=O(r_s^3)+O(r_s^4),\]
\[|\mathbb{E}  \left[\bar{P}_{ikj}\bar{P}_{ilj}\right]|=\mathbb{E}  \left[A_{ik}A_{kj}A_{il}A_{jl}\right]+O(r_s^4)\leq \mathbb{E}  \left[A_{ik}A_{jk}A_{jl}\right]+O(r_s^4)=O(r_s^3)+O(r_s^4),\]
\[|\mathbb{E}  \left[\bar{P}_{ijk}\bar{P}_{lij}\right]|=\mathbb{E}  \left[A_{ij}A_{jk}A_{li}\right]+O(r_s^4)=O(r_s^3)+O(r_s^4),\]
\[\mathbb{E}  \left[\bar{P}_{ijk}^2\right]=\mathbb{E}  \left[A_{ij}A_{jk}\right]+O(r_s^4)=O(r_s^2),\]
\[|\mathbb{E}  \left[\bar{P}_{ijk}\bar{P}_{ikj}\right]|\leq \mathbb{E}  \left[A_{ij}A_{jk}\right]+O(r_s^3)=O(r_s^2).\]
Therefore, we get
\begin{align*}
\mathbb{E}  \left[P_n^2\right]= O\left(n^4 r_s^3\right) + O  \left(n^3 r_s^2\right).
\end{align*}
By \eqref{exp2pathe} and the assumption $nr_s = \omega(1)$, we have 
$\sum_{i \ne j \ne k} \mathbb{E}[A_{ij}A_{ik}] = \Theta(n^3r_s^2)$. 
By Markov's inequality, we obtain
\[\mathbb{P}\left(\left|\frac{P_n}{\sum_{i \ne j \ne k} \mathbb{E}[A_{ij}A_{ik}]}\right|>\epsilon\right)\leq \frac{1}{\epsilon^2}\frac{\mathbb{E}\left[P_n^2\right]}{(\sum_{i \ne j \ne k} \mathbb{E}[A_{ij}A_{ik}])^2}=O\left(\frac{1}{n^2r_n^2}\right)=o(1).\]
Then (\ref{thempreq1}) holds.

Next, we derive the leading-order term of the sum of expectations on the 
right-hand side of \eqref{thempreq2} and establish \eqref{thempreq2} 
by applying Markov's inequality.

If $r_d \leq r_s < 2r_d$, by Proposition \ref{proptria}, we have
\begin{align}\nonumber
\sum_{i\ne j\ne k} \mathbb{E}  \left[A_{ij}A_{jk}A_{ki}\right]
  &=\bigl(\tau n (\tau n-1)(\tau n-2)+ (1-\tau)n((1-\tau)n-1)((1-\tau)n-2)\bigr)3r_s^2
   \\ \nonumber
  &\quad+3\bigl(4r_s r_d - r_s^2\bigr)
    \Bigl[\tau n (\tau n-1)(1-\tau)n+\tau n(1-\tau)n((1-\tau)n-1)\Bigr]  \\ \label{etrian}
  &= \Big[\bigl(1-3\tau(1-\tau)\bigr) 3r_s^2
    + 3(1-\tau)\tau \bigl(4r_s r_d - r_s^2\bigr)\Big]n^3+O(n^2r_s^2).
\end{align}
 If $r_s \geq 2r_d$, by Proposition \ref{proptria}, we have
\begin{align}\nonumber
\sum_{i\ne j\ne k} \mathbb{E}  \left[A_{ij}A_{jk}A_{ki}\right]
  &=\bigl(\tau n (\tau n-1)(\tau n-2)+ (1-\tau)n((1-\tau)n-1)((1-\tau)n-2)\bigr)3r_s^2
   \\ \nonumber
  &\quad+ 3
    \Bigl[\tau n (\tau n-1)(1-\tau)n+\tau n(1-\tau)n((1-\tau)n-1)\Bigr]4r_d^2  \\ \label{etrian2}
  &= \Big[\bigl(1-3\tau(1-\tau)\bigr) 3r_s^2
    + 12(1-\tau)\tau r_d^2\Big]n^3+O(n^2r_s^2).
\end{align}

Now, we establish \eqref{thempreq2} by applying Markov's inequality. 
Specifically, we derive an upper bound for the variance of the 
left-hand side of \eqref{thempreq2}. Let $T_n = \sum_{i \ne j \ne k} \bar{\Delta}_{ijk}$ with $\bar{\Delta}_{ijk}
= A_{ij}A_{jk}A_{ki} -\mathbb{E}  \left[A_{ij}A_{jk}A_{ki}\right]$. Then
\begin{align*}
\mathbb{E}  \left[ T_n^2 \right]= \sum_{i \ne j \ne k} \sum_{i_1 \ne j_1 \ne k_1}\mathbb{E}  \left[ \bar{\Delta}_{ijk} \bar{\Delta}_{i_1 j_1 k_1} \right] 
\end{align*}

If $\{i,j,k\}$ and $\{i_1,j_1,k_1\}$ are disjoint, then  $\bar{\Delta}_{ijk}$ and $\bar{\Delta}_{i_1 j_1 k_1}$ are independent. Hence, the $\mathbb{E}  \left[\bar{\Delta}_{ijk} \bar{\Delta}_{i_1 j_1 k_1} \right]=\mathbb{E}  \left[\bar{\Delta}_{ijk} \right]\mathbb{E}  \left[ \bar{\Delta}_{i_1 j_1 k_1} \right]=0$. 

Next, we show if $\{i,j,k\}$ and $\{i_1,j_1,k_1\}$ have exactly one vertex in common,  $\mathbb{E}  \left[\bar{\Delta}_{ijk} \bar{\Delta}_{i_1 j_1 k_1}\right]=0$.  To this end, we only need to show  $\mathbb{E}  \left[\bar{\Delta}_{123} \bar{\Delta}_{145}\right]=0$. Note that $\mathbb{E}  \left[ \bar{\Delta}_{123} \mid X_1 \right]=0$ by Proposition \ref{proptria}. Then
\[
\mathbb{E}  \left[ \bar{\Delta}_{123} \bar{\Delta}_{145} \right]
= \mathbb{E}  \left[
\mathbb{E}  \left[ \bar{\Delta}_{123} \bar{\Delta}_{145} \mid X_1 \right]
\right] = \mathbb{E}  \left[
\mathbb{E}  \left[ \bar{\Delta}_{123} \mid X_1 \right]
\mathbb{E}  \left[ \bar{\Delta}_{145} \mid X_1 \right]
\right]=0.
\]
Therefore, we get
\begin{align*}
\mathbb{E}  \left[T_n^2\right]
&= 48 \sum_{i\ne j\ne k\ne l}
\mathbb{E}  \left[\bar{\Delta}_{ijk} \bar{\Delta}_{ijl}\right]
+ \sum_{i\ne j\ne k} \mathbb{E}  \left[\bar{\Delta}_{ijk}^2\right]. 
\end{align*}
By Proposition \ref{proptria}, it is easy to obtain
\begin{align*}
\sum_{i\ne j\ne k} \mathbb{E}  \left[\bar{\Delta}_{ijk}^2\right]
&\leq \sum_{i\ne j\ne k} (\mathbb{E}  \left[A_{ij}A_{jk}A_{ki}\right]+O(r_s^4))= O  \left(n^3 r_s^2\right), \\
|\mathbb{E}  \left[\bar{\Delta}_{123} \bar{\Delta}_{124}\right]|
&\leq \mathbb{E}  \left[A_{12}A_{23}A_{31} A_{14}A_{24}\right] + \mathbb{E}  \left[\Delta_{123}\right]\mathbb{E}  \left[\Delta_{124}\right] \\
&\leq \mathbb{E}  \left[A_{12}A_{23}A_{24}\right]+O(r_n^4)\\
&= \Theta  \left(r_s^3\right).
\end{align*}
Therefore,
\begin{align*}
\mathbb{E}  \left[T_n^2\right]=O( n^4 r_s^3) + O(n^3 r_s^2).
\end{align*}
Then asymptotic relation \eqref{thempreq2} follows from Markov's 
inequality, \eqref{etrian}, \eqref{etrian2}, and the assumption $nr_s = \omega(1)$. Combining \eqref{thempreq1}, 
\eqref{thempreq2}, \eqref{etrian}, and \eqref{etrian2}, the proof of 
Theorem \ref{asygc} is complete.

\qed

\subsection{Proof of Proposition \ref{propcorunbg}} 
The expression (\ref{glambdta}) follows directly from Theorem \ref{asygc}. We need only identify the monotonic intervals for \(g(\lambda, \tau)\).

Suppose $\lambda\in[1,2)$. Then simple  algebra yields
\begin{equation*}
g(\lambda, \tau)= \frac{
    \bigl[3 - 12\,\tau(1-\tau)\bigr]\lambda^2 + 12\,\tau(1-\tau)\,\lambda
  }{
    \bigl[4 - 12\,\tau(1-\tau)\bigr]\lambda^2
    + 8\,\tau(1-\tau)\,\lambda + 4\,\tau(1-\tau)
  }.
\end{equation*}
Fix $\tau\in(0,1)$ and let $a=\tau(1-\tau)$.
Then $g(\lambda, \tau)$ can be written as
\[
g(\lambda, \tau)=\frac{(3-12a)\lambda^2+12a\lambda}{(4-12a)\lambda^2+8a\lambda+4a}.
\]
Set
\[
N(\lambda)=(3-12a)\lambda^2+12a\lambda,\qquad
D(\lambda)=(4-12a)\lambda^2+8a\lambda+4a.
\]
Then the derivative of $g(\lambda, \tau)$ with respect to $\lambda$ is equal to
\[
g'(\lambda, \tau)=\frac{N'(\lambda)D(\lambda)-N(\lambda)D'(\lambda)}{[D(\lambda)]^2}.
\]
Straightforward calculation yields
\[
N'(\lambda)=2(3-12a)\lambda+12a,
\ \ \ \ \ 
D'(\lambda)=2(4-12a)\lambda+8a.
\]
Then
\[
N'(\lambda)D(\lambda)-N(\lambda)D'(\lambda)
=
24a(\lambda-1)(2a\lambda-2a-\lambda).
\]
Since $0<a=\tau(1-\tau)\leq \frac{1}{4}$ and $\lambda\geq1$, then $2a\lambda-2a-\lambda<0$. Then $g'(\lambda, \tau)<0$, which implies $g(\lambda, \tau)$ is decreasing for $\lambda\in[1,2)$. Then $g(\lambda, \tau)<g(1, \tau)=\frac{3}{4}$ for $\lambda\in(1,2)$.

Suppose $\lambda \geq 2$. Then
\begin{equation*}
g(\lambda, \tau) = \frac{
    \bigl[3 - 9\,\tau(1-\tau)\bigr]\lambda^2 + 12\,\tau(1-\tau)
  }{
    \bigl[4 - 12\,\tau(1-\tau)\bigr]\lambda^2 + 8\,\tau(1-\tau)\,\lambda + 4\,\tau(1-\tau)
  }
\end{equation*}
Let $u=\tau(1-\tau)$. Then
\[
g(\lambda, \tau)=\frac{(3-9u)\lambda^2+12u}{(4-12u)\lambda^2+8u\lambda+4u}.
\]
Fix $\tau\in(0,1)$. The derivative of $g(\lambda, \tau)$ with respect to $\lambda$ is equal to
\[
g'(\lambda,\tau)=
\frac{24u\Bigl[(1-3u)\lambda^2-3(1-3u)\lambda-4u\Bigr]}
{\Bigl((4-12u)\lambda^2+8u\lambda+4u\Bigr)^2}.
\]
Since \(0\le u\le \frac14\), we have $1-3u>0$.
Then \(q(\lambda)=(1-3u)\lambda^2-3(1-3u)\lambda-4u\) is an upward-opening quadratic function. Its roots are
\[
\lambda=
\frac{3(1-3u)\pm\sqrt{(1-3u)(9-11u)}}{2(1-3u)}.
\]
Note that the product of the roots is
\[
\frac{-4u}{1-3u}< 0.
\]
Therefore, one root is negative and the other is positive. Moreover,
\[
q(2)=4(1-3u)-6(1-3u)-4u=-2+2u<0,
\]
which implies $2$ is less than the positive root.
Hence, the unique critical point in \((2,\infty)\) is
\[
\lambda^*=
\frac{3(1-3u)+\sqrt{(1-3u)(9-11u)}}{2(1-3u)}.
\]
Then, we have
\[
g'(\lambda,\tau)<0 \quad \text{for } 2<\lambda<\lambda_*,
\]
and
\[
g'(\lambda,\tau)>0 \quad \text{for } \lambda>\lambda_*.
\]
It follows that  \(g(\lambda,\tau)\) is decreasing on \((2,\lambda_*)\) and increasing on \((\lambda_*,\infty)\). Thus,  \(g(\lambda,\tau)\)  attains its minimum at
\[
\lambda^* = \frac{3}{2} + \frac{1}{2} \sqrt{\frac{9 - 11\tau(1-\tau)}{1 - 3\tau(1-\tau)}}.
\]
Then the proof is complete.

 \qed

\subsection{Proof of Theorem  \ref{asyalc}}

 Without loss of generality, let $V_1=\{1,2,\dots,\tau n\}$ and $V_2=\{\tau n+1,\tau n+2,\dots, n\}$. By definition,  the average clustering coefficient can be written as
 \begin{equation} \label{avccrw1}
\overline{\mathcal{C}}_n=\frac{1}{n}\sum_{i=1}^{\tau n}\frac{\sum_{j\neq k} A_{ij}A_{jk}A_{ki}}{d_i(d_i-1)}\mathbb{I}[d_i\geq2]+\frac{1}{n}\sum_{i=\tau n+1}^{n}\frac{\sum_{j\neq k} A_{ij}A_{jk}A_{ki}}{d_i(d_i-1)}\mathbb{I}[d_i\geq2].
\end{equation}
Next, we determine the limit of each term in \eqref{avccrw1}. 

We begin by 
evaluating the limit of the first term. By Lemma 3.1 in \cite{Y25c}, $\mathbb{E}[A_{ij}] = 2r_s$ if nodes $i$ 
and $j$ are in the same community, and $\mathbb{E}[A_{ij}] = 2r_d$ 
if they are in different communities.
Let $\mu_1=\mathbb{E}[d_i]$ for $i\in V_1$ and $\mu_2=\mathbb{E}[d_i]$ for $i\in V_2$.
Then
\begin{align}\label{mu1exp}
  \mu_1=\mathbb{E}[d_i] &= (2r_s)(\tau n-1) + 2r_d(1-\tau)n= 2n\bigl[(1-\tau)r_d + \tau r_s\bigr]-2r_s,\\ \label{mu2exp}
  \mu_2=\mathbb{E}[d_i] &= (2r_s)((1-\tau)n-1) + 2r_d \tau n= 2n\bigl[(1-\tau)r_s + \tau r_d\bigr]-2r_s.
\end{align}
For convenience, let $\Delta_i=\sum_{j\neq k} A_{ij}A_{jk}A_{ki}$. Then the first term of (\ref{avccrw1}) can be expressed as
\begin{align}\label{avccrw2}
  \frac{1}{n}\sum_{i=1}^{\tau n} \frac{\Delta_i\mathbb{I}[d_i\geq2]}{d_i(d_i - 1)}
  = \frac{1}{n}\sum_{i=1}^{\tau n} \frac{\Delta_i\mathbb{I}[d_i\geq2]}{\mu_1(\mu_1 - 1)}+ \frac{1}{n}\sum_{i=1}^{\tau n}
    \frac{\Delta_i\bigl(\mu_1(\mu_1-1) - d_i(d_i-1)\bigr)}{d_i(d_i-1) \mu_1(\mu_1-1)}\mathbb{I}[d_i\geq2].
\end{align}
We will derive the limit of the first term of (\ref{avccrw2}) and show the second term is negligible.  

Note that $\Delta_i=0$ if $d_i=0, 1$. Hence, $\Delta_i\mathbb{I}[d_i\geq2]=\Delta_i$.
By a similar argument as in the proof of (\ref{thempreq2}), we have
\begin{equation*}
  \sum_{i=1}^{\tau n}\Delta_i =\sum_{i=1}^{\tau n}\mathbb{E}[\Delta_i]+O_P \left(\sqrt{n^4 r_s^3}\right).
\end{equation*}
Recall that $r_s = \Theta(r_d)$ by assumption, which implies that 
$\mu_1 = \Theta(nr_s)$. It follows that:
\begin{equation}\label{avccrw3}
  \frac{1}{n}\sum_{i=1}^{\tau n} \frac{\Delta_i\mathbb{I}[d_i\geq2]}{\mu_1(\mu_1-1)}=\frac{1}{n}\sum_{i=1}^{\tau n} \frac{\Delta_i}{\mu_1(\mu_1-1)}
  = \tau \frac{\mathbb{E}[\Delta_1]}{\mu_1(\mu_1-1)}
  + O_P  \left(\frac{1}{n\sqrt{r_n}}\right).
\end{equation}

 Next, we show the second term of (\ref{avccrw2}) is negligible.   For nodes $i,j$ in the same community,  $A_{ij}=\mathbb{I}[d(X_i,X_j)\leq r_s]\geq \mathbb{I}[d(X_i,X_j)\leq r_d]$. Then $d_i=\sum_{j}A_{ij}\geq \sum_{j\neq i}\mathbb{I}[d(X_i,X_j)\leq r_d]$.  Note that $\mathbb{E}[\mathbb{I}[d(X_i,X_j)\leq r_d]|X_i]=2r_d$ by Lemma 3.1 in \cite{Y25c}. Given $X_i$, $\sum_{j\neq i}\mathbb{I}[d(X_i,X_j)\leq r_d]$ follows binomial distribution with parameter $n-1$ and success probability $2r_d$. Let $M$ be a large positive constant. By the  Chernoff bound for binomial distribution, we have
\begin{align}\label{chernoff1}
     \mathbb{P}  \left(d_i \leq \frac{n r_d}{M}\Big|X_i\right)\leq \mathbb{P}  \left(\sum_{j\neq i}\mathbb{I}[d(X_i,X_j)\leq r_d] \leq \frac{n r_d}{M}\Big|X_i\right) \leq e^{-c_1 n r_d}.
\end{align}

For nodes $i,j$ in different communities,  $A_{ij}=\mathbb{I}[d(X_i,X_j)\leq r_d]\leq \mathbb{I}[d(X_i,X_j)\leq r_s]$. Then $d_i=\sum_{j}A_{ij}\leq \sum_{j\neq i}\mathbb{I}[d(X_i,X_j)\leq r_s]$.  Note that $\mathbb{E}[\mathbb{I}[d(X_i,X_j)\leq r_s]|X_i]=2r_s$ by Lemma 3.1 in \cite{Y25c}. Given $X_i$, $\sum_{j\neq i}\mathbb{I}[d(X_i,X_j)\leq r_s]$ follows binomial distribution with parameter $n-1$ and success probability $2r_s$. For the same large positive constant $M$ given in the previous case, by the  Chernoff bound for binomial distribution, we have
\begin{align}\label{chernoff2}
     \mathbb{P}  \left(d_i \geq M n r_s\Big|X_i\right)\leq \mathbb{P}  \left(\sum_{j\neq i}\mathbb{I}[d(X_i,X_j)\leq r_s] \geq  M n r_s\Big|X_i\right) \leq e^{-c_2 n r_s}.
\end{align}

We claim that $\mathbb{E}[\Delta_1^2] = O(n^4r_s^2)$. To see this, observe 
that for distinct indices $i, j, i_1, j_1$, Proposition \ref{proptria} 
and the assumption $r_s = \Theta(r_d)$ imply that:
\[\mathbb{E}[A_{1i}A_{1j}A_{ij}A_{1i_1}A_{1j_1}A_{i_1j_1}]=\mathbb{E}\big[\mathbb{E}[A_{1i}A_{1j}A_{ij}|X_1]\mathbb{E}[A_{1i_1}A_{1j_1}A_{i_1j_1}|X_1]\big]=O(r_s^4),\]
\[\mathbb{E}[A_{1i}A_{1j}A_{ij}A_{1i}A_{1j_1}A_{ij_1}]\leq\mathbb{E}\big[\mathbb{E}[A_{1i}A_{1j}A_{ij}|X_1]\mathbb{E}[A_{1j_1}|X_1]\big]=O(r_s^3),\]
\[\mathbb{E}[A_{1i}A_{1j}A_{ij}]=O(r_s^2).\]
Since $nr_n=\omega(1)$, then
\begin{align*}
   \mathbb{E}[\Delta_1^2]&=\sum_{i\neq j\neq i_1\neq j_1} \mathbb{E}[A_{1i}A_{1j}A_{ij}A_{1i_1}A_{1j_1}A_{i_1j_1}]+C_1\sum_{i\neq j\neq j_1} \mathbb{E}[A_{1i}A_{1j}A_{ij}A_{1i}A_{1j_1}A_{ij_1}]\\
   &\quad+C_2\sum_{i\neq j} \mathbb{E}[A_{1i}A_{1j}A_{ij}]\\
   &=O(n^4r_s^4+n^3r_s^3+n^2r_s^2)=O(n^4r_s^4).
\end{align*}

Recall that $r_s=\Theta(r_d)$ and $nr_s=\omega(1)$ by assumption, and $\mu_1=\Theta(nr_s)$.
By the Cauchy-Schwarz inequality and  (\ref{chernoff1}), we get
\begin{align}\nonumber
  &\mathbb{E}  \left[
    \frac{1}{n}\sum_{i=1}^{\tau n}
    \left|
      \frac{\Delta_i\bigl(\mu_1(\mu_1-1) - d_i(d_i-1)\bigr)}{d_i(d_i-1) \mu_1(\mu_1-1)}
    \right|
    \mathbb{I}  \left [2 \leq d_i \leq \frac{n r_d}{M}\right]
  \right]\\ \nonumber
  &\leq O(1)\mathbb{E}  \left[
        \frac{1}{n}\sum_{i=1}^{\tau n} \Delta_i \mathbb{I}   \left[2 \leq d_i \leq \frac{n r_d}{M}\right]
    \right] \\ \nonumber
  &\leq O(1)  \frac{1}{n}\sum_{i=1}^{\tau n} \sqrt{\mathbb{E}[\Delta_i^2]
    \mathbb{P}    \left(d_i \leq \frac{n r_d}{M}\right)}\\  \nonumber
  &= O  \left(\sqrt{n^4 r_s^4}e^{-\frac{c_1 n r_d}{2}}\right)\\ \label{remeq1}
  &=o(1).
\end{align}
By a similar argument and (\ref{chernoff2}), we have
\begin{align}\nonumber
  &\mathbb{E}  \left[
    \frac{1}{n}\sum_{i=1}^{\tau n}
    \left|
      \frac{\Delta_i\bigl(\mu_1(\mu_1-1) - d_i(d_i-1)\bigr)}{d_i(d_i-1) \mu_1(\mu_1-1)}
    \right|
    \mathbb{I}  \left [d_i \geq M n r_s\right]
  \right]\\ \nonumber
  &\leq O(1)  \frac{1}{n}\sum_{i=1}^{\tau n} \sqrt{\mathbb{E}[\Delta_i^2]
    \mathbb{P}    \left(d_i \geq M n r_s\right)}\\  \nonumber
  &= O  \left(\sqrt{n^4 r_s^4}e^{-\frac{c_2 n r_s}{2}}\right)\\ \label{remeq2}
  &=o(1).
\end{align}

Now we consider the case $\dfrac{n r_d}{M} \leq d_i \leq M  n r_s$. We claim $\mathbb{E}[(d_1-\mu_1)^2]=O(nr_s)$. To see this,
note that $d_1-\mu_1=\sum_{j=2}^{\tau n}(A_{1j}-2r_s)+\sum_{j=\tau n+1}^{n}(A_{1j}-2r_d)$. By Lemma 3.1 in \cite{Y25c}, $\mathbb{E}[A_{1i}|X_1]=2r_s$ for $2\leq i\leq \tau n$.  For $2\leq i\neq j\leq \tau n$, we have
\[\mathbb{E}[(A_{1i}-2r_s)(A_{1j}-2r_s)|X_1]=\mathbb{E}[(A_{1i}-2r_s)|X_1]\mathbb{E}[(A_{1j}-2r_s)|X_1]=0.\]
Then $\mathbb{E}[(A_{1i}-2r_s)(A_{1j}-2r_s)]=0$ and
\[\mathbb{E}\left[\left(\sum_{j=2}^{\tau n}(A_{1j}-2r_s)\right)^2\right]=\mathbb{E}\left[\sum_{j=2}^{\tau n}(A_{1j}-2r_s)^2\right]=O(nr_s).\]
The second moment of the second term of  $d_1-\mu_1$ can be similarly bounded. Therefore,  $\mathbb{E}[(d_1-\mu_1)^2]=O(nr_s)$.

 Straightforward calculation yields
\begin{align*}
  \mu_1(\mu_1-1) - d_i(d_i-1)
  = \mu_1^2 - d_i^2 + d_i - \mu_1= (\mu_1 - d_i)\bigl[\mu_1 + d_i - 1\bigr].
\end{align*} 
Since $nr_s=\omega(1)$ by assumption, then
\begin{align}\nonumber
  &\mathbb{E}  \left[\frac{1}{n}
    \sum_{i=1}^{\tau n}
    \left|
      \frac{\Delta_i\bigl(\mu_1(\mu_1-1) - d_i(d_i-1)\bigr)}{d_i(d_i-1) \mu_1(\mu_1-1)}
    \right|
    \cdot  \mathbb{I}  \left[\frac{n r_d}{M} \leq d_i \leq M  n r_s\right]
  \right]
  \\ \nonumber
  &\leq \mathbb{E}  \left[
      \frac{1}{n}\sum_{i=1}^{\tau n}
      \frac{\Delta_i |\mu_i - d_i| |\mu_i + d_i - 1|}{d_i(d_i-1) \mu_1(\mu_1-1)}
      \cdot  \mathbb{I}  \left[\frac{n r_s}{M} \leq d_i \leq M  n r_s\right]
    \right] \\  \nonumber
  &= O  \left(\frac{1}{n^4 r_s^3}\right)
    \sum_{i=1}^{\tau n} \mathbb{E}\bigl[\Delta_i |\mu_i - d_i|\bigr] \\\nonumber
  &= O  \left(\frac{n}{n^4 r_s^3}\right)
    \sqrt{\mathbb{E}[\Delta_1^2]\;\mathbb{E}[(\mu_1-d_1)^2}] \\ \label{remeq3}
  &= O\left(\
    \frac{\sqrt{n^4 r_n^4\cdot n r_n}}{n^3 r_n^3}\right)
  = O  \left(\frac{1}{\sqrt{n r_n}}\right)=o(1).
\end{align}

Combining (\ref{remeq1}), (\ref{remeq2}), (\ref{remeq3}) and Markov's inequality yields
\begin{align}\label{remeq4}
\frac{1}{n}\sum_{i=1}^{\tau n}
    \frac{\Delta_i\bigl(\mu_1(\mu_1-1) - d_i(d_i-1)\bigr)}{d_i(d_i-1) \mu_1(\mu_1-1)}\mathbb{I}[d_i\geq2]=o_P(1).
\end{align}

By (\ref{avccrw2}), (\ref{avccrw3}) and (\ref{remeq4}), we have
\begin{align}\label{remeq5}
  \frac{1}{n}\sum_{i=1}^{\tau n}\frac{\sum_{j\neq k} A_{ij}A_{jk}A_{ki}}{d_i(d_i-1)}\mathbb{I}[d_i\geq2]
  = \tau \frac{\mathbb{E}[\Delta_1]}{\mu_1(\mu_1-1)}+o_P(1).
\end{align}

By a similar argument, it is easy to get
\begin{align}\label{remeq6}
 \frac{1}{n}\sum_{i=\tau n+1}^{n}\frac{\sum_{j\neq k} A_{ij}A_{jk}A_{ki}}{d_i(d_i-1)}\mathbb{I}[d_i\geq2]
  = (1-\tau) \frac{\mathbb{E}[\Delta_n]}{\mu_2(\mu_2-1)}+o_P(1).
\end{align}

Combining (\ref{avccrw1}), (\ref{remeq5}) and (\ref{remeq6}) yields
 \begin{equation} \label{1avccrw1}
\overline{\mathcal{C}}_n=\tau \frac{\mathbb{E}[\Delta_1]}{\mu_1(\mu_1-1)}+(1-\tau) \frac{\mathbb{E}[\Delta_n]}{\mu_2(\mu_2-1)}+o_P(1).
\end{equation}
Suppose   $r_d \leq r_s < 2r_d$. By Proposition \ref{proptria}, we have
\begin{align}\nonumber
  \mathbb{E}[\Delta_1]
  &= 3r_s^2 (\tau n-1)(\tau n-2) + (4r_s r_d - r_s^2)
     \bigl[2(\tau n-1)(1-\tau)n + (1-\tau)n((1-\tau)n-1)\bigr] \\ \label{edelta1}
  &= n^2\Bigl[3\tau^2 r_s^2 + (4r_s r_d - r_s^2)(1-\tau^2)\Bigr]+O(nr_s^2),
\end{align}
\begin{align}\nonumber
  \mathbb{E}[\Delta_n]  &= ((1-\tau)n-1)((1-\tau)n-2) 3r_s^2
    + (4r_s r_d - r_s^2)\bigl[2\tau n((1-\tau) n-1) + \tau n(\tau n-1)\bigr] \\ \nonumber
  &= (1-\tau)^2 n^23r_s^2
    + (4r_s r_d - r_s^2)\bigl[2\tau(1-\tau)n^2 + \tau^2 n^2\bigr]+O(nr_s^2)\\ \label{edelta2}
  &= n^2\Bigl[3(1-\tau)^2 r_s^2 + (4r_s r_d - r_s^2) \tau(2-\tau)\Bigr]+O(nr_s^2).
\end{align}

Suppose $r_s \geq 2r_d$. By Proposition \ref{proptria}, we have
\begin{align}\nonumber
  \mathbb{E}[\Delta_1]
  &= 3r_s^2 (\tau n-1)(\tau n-2) + 4 r_d^2
     \bigl[2(\tau n-1)(1-\tau)n + (1-\tau)n((1-\tau)n-1)\bigr] \\ \label{edelta3}
  &= n^2\Bigl[3\tau^2 r_s^2 + 4 r_d^2(1-\tau^2)\Bigr]+O(nr_s^2),
\end{align}
\begin{align}\nonumber
  \mathbb{E}[\Delta_n]  &= ((1-\tau)n-1)((1-\tau)n-2) 3r_s^2
    + 4 r_d^2\bigl[2\tau n((1-\tau) n-1) + \tau n(\tau n-1)\bigr] \\ 
 \label{edelta4}
  &= n^2\Bigl[3(1-\tau)^2 r_s^2 + 4 r_d^2\tau(2-\tau)\Bigr]+O(nr_s^2).
\end{align}

Combining (\ref{mu1exp}), (\ref{mu2exp}), (\ref{1avccrw1}), (\ref{edelta1}),  (\ref{edelta2}) , (\ref{edelta3}) and  (\ref{edelta4})  yields  the results of Theorem \ref{asyalc}. Then the proof is complete.

\qed

\subsection{Proof of Proposition \ref{propml1}}

	{\bf Proof of (a):} Fix $\tau \in (0,1)$, and denote $ h(\lambda)=h(\lambda,\tau)$. 
	For $ \lambda \in [1, 2)$, the derivative of $h(\lambda)$ is
	\begin{align*}
		h'(\lambda) 
		& = \frac{\tau (1-\tau)}{2} \bigg[ \frac{(2\tau^2-2\tau-1)\lambda + 2(1-\tau^2)}{(\tau\lambda + 1-\tau)^3} + \frac{(2\tau^2-2\tau-1)\lambda + 2\tau(2 - \tau)}{((1-\tau)\lambda + \tau)^3} \bigg].
	\end{align*}
	Observe that $ h'(1) = 0 $. Hence, $\lambda^* = 1$ is a critical point. The second derivative of $h(\lambda)$ is
	\begin{align*}
		h''(\lambda) 
		& = - \frac{\tau (1-\tau)}{2} \bigg[\frac{2\tau(2\tau^2-2\tau-1)\lambda + (1-\tau)(4\tau^2+8\tau+1)}{(\tau\lambda + 1-\tau)^4} \\ 
		& \qquad \qquad \qquad \
		+ \frac{2(1-\tau)(2\tau^2-2\tau-1)\lambda + \tau(4\tau^2-16\tau+13)}{((1-\tau)\lambda + \tau)^4} \bigg].
	\end{align*}
	It follows that
	\begin{align*}
		h''(1) 
		& = \frac{\tau (1-\tau)(16\tau^2 - 16\tau + 1)}{2}.
	\end{align*}
	By the graph of $ f(\tau) = 16\tau^2 - 16\tau + 1$, $ h''(1) > 0 $ if $\tau \in (0, 0.0670) \cup (0.9330, 1)$.
	This implies $ h(\lambda) $ is convex around $\lambda^* = 1$, i.e. $ h(\lambda) $ is increasing on the right-hand neighborhood of $\lambda^* = 1$.
	Hence, for $ \lambda \in [1, 2)$ and fixed $\tau \in (0, 0.0670) \cup (0.9330, 1)$, the function $h(\lambda,\tau)$ defined in Corollary \ref{lgccor2} has at least one local maximum which is greater than $h(1) = 3/4$ and occurs at $\lambda^* \in (1, 2)$.
	In addition, $h'(\lambda) = 0$ is equivalent to the following quartic equation
	\begin{align*}
		[(2\tau^2-2\tau-1)\lambda + 2(1-\tau^2)]((1-\tau)\lambda + \tau)^3 + [(2\tau^2-2\tau-1)\lambda + 2\tau(2 - \tau)](\tau\lambda + 1-\tau)^3 = 0,
	\end{align*}
	which has at most four real solutions.
	Hence, there are at most two local maximums greater than $3/4$, occurring at $\lambda^* \in (1, 2)$.

{\bf Proof of (b): }
Let $\tau = 0.01$. For $ \lambda \in [1, 2)$, we have
	\begin{align*}
		h(\lambda) 
		& = 0.03\frac{- 833\lambda^2 + 3333\lambda }{(\lambda + 99)^2} + 0.99\frac{7301\lambda^2 + 199\lambda}{(99\lambda + 1)^2}, \\
		h'(\lambda) 
		& = 0.99\frac{(-5099\lambda + 9999)(99\lambda + 1)^3 + (-5099\lambda + 199)(\lambda + 99)^3}{(\lambda + 99)^3(99\lambda + 1)^3}, \\
		h''(\lambda) 
		& = 1.98\bigg[\frac{5099\lambda - 267399}{(\lambda + 99)^4} + \frac{504801\lambda - 32101}{(99\lambda + 1)^4}\bigg].
	\end{align*}
	Let $ h'(\lambda) = 0 $.
	The solutions to the equation
	\begin{align*}
		(-5099\lambda + 9999)(99\lambda + 1)^3 + (-5099\lambda + 199)(\lambda + 99)^3 = 0
	\end{align*}
	are $\lambda_1 = -0.6466$, $\lambda_2 = 0.0393$, $\lambda_3 = 1$, and $\lambda_4 = 1.5377$.
	In $ [1, 2) $, the critical points are
	\begin{align*}
		\lambda^*_1 = 1 \quad \text{and} \quad \lambda^*_2 = 1.5377.	
	\end{align*}
	Note that
	\begin{align*}
		h''(\lambda^*_1) & = h''(1) = 0.0042 > 0, \\
		h''(\lambda^*_2) & = h''(1.5377) = -0.0024 < 0.
	\end{align*}
	Then $h(\lambda)$ has a local minimum at $1$ and a local maximum at $1.5377$, i.e. it is increasing on $(1, 1.5377)$ and decreasing on $(1.5377, 2)$.
	Hence, the local minimum is
	\begin{align*}
		h(\lambda^*_1) = h(1) = \frac{3}{4},
	\end{align*}
	and the local maximum is
	\begin{align*}
		h(\lambda^*_2) = h(1.5377) = 0.7502.
	\end{align*}
	For $ \lambda \in [2, \infty)$, we have
	\begin{align*}
		h(\lambda) 
		& = 0.0075\frac{\lambda^2 + 13332}{(\lambda + 99)^2} + 0.2475\frac{29403\lambda^2 + 796}{(99\lambda + 1)^2}, \\
		h'(\lambda) 
		& = 0.495\frac{(3\lambda - 404)(99\lambda + 1)^3 + 99(297\lambda - 796)(\lambda + 99)^3}{(\lambda + 99)^3(99\lambda + 1)^3}, \\
		h''(\lambda) 
		& = 1.485\bigg[\frac{-2\lambda + 503}{(\lambda + 99)^4} + 9801\frac{-198\lambda + 797}{(99\lambda + 1)^4}\bigg].
	\end{align*}
	Let $ h'(\lambda) = 0 $.
	The real solutions to the equation
	\begin{align*}
		(3\lambda - 404)(99\lambda + 1)^3 + 99(297\lambda - 796)(\lambda + 99)^3 = 0
	\end{align*}
	are $\lambda_1 = -8.3513$, $\lambda_2 = 3.0224$, $\lambda_3 = 8.0742$, and $\lambda_4 = 127.6016$.
	In $ [2, \infty) $, the critical points are
	\begin{align*}
		\lambda^*_3 = 3.0224, \quad \lambda^*_4 = 8.0742, \quad \text{and} \quad \lambda^*_5 = 127.6016.	
	\end{align*}
	Note that
	\begin{align*}
		h''(\lambda^*_3) & = h''(3.0224) = 0.0004 > 0, \\
		h''(\lambda^*_4) & = h''(8.0742) = -2.2937 \times 10^{-5} < 0, \\
		h''(\lambda^*_5) & = h''(127.6016) = 1.2558 \times 10^{-7} > 0.
	\end{align*}
	Then $h(\lambda)$ has two local minimums at $3.0224$ and $127.6016$, and a local maximum at $8.0742$, i.e. it is decreasing on $(2, 3.0224)$ and $(8.0742, 127.6016)$, and increasing on $(3.0224, 8.0742)$ and $(127.6016, \infty)$.
	Hence, the local minimums are
	\begin{align*}
		h(\lambda^*_3) & = h(3.0224) = 0.7494, \\
		h(\lambda^*_5) & = h(127.6016) = 0.7467,
	\end{align*}
	and the local maximum is
	\begin{align*}
		h(\lambda^*_4) = h(8.0742) = 0.7497.
	\end{align*}
	Comparing all of the local extrema and $\displaystyle \lim_{\lambda \to \infty} h(\lambda) = 3/4$, the global minimum is
	\begin{align*}
		h_{\min}(\lambda) = h(127.6016) = 0.7467,
	\end{align*}
	and the global maximum is
	\begin{align*}
		h_{\max}(\lambda) = h(1.5377) = 0.7502 \ > \ 3/4.
	\end{align*}
	In addition, observe that $\displaystyle \lim_{\lambda \to 2^-} h(\lambda) = h(2) $, i.e. $h(\lambda)$ is continuous at $2$. 
	Hence, $ h(\lambda) $ is increasing on $(1, 1.5377)$, $(3.0224, 8.0742)$, and $(127.6016, \infty)$, and decreasing on $(1.5377, 3.0224)$ and $(8.0742, 127.6016)$.

{\bf Proof of (c):} Fixed $\tau = 0.1$. For $ \lambda \in [1, 2)$, we have
	\begin{align*}
		h(\lambda) 
		& = 0.3\frac{-8\lambda^2 + 33\lambda}{(\lambda + 9)^2} + 0.9\frac{56\lambda^2 + 19\lambda}{(9\lambda + 1)^2}, \\
		h'(\lambda) 
		& = 0.9 \frac{(- 59\lambda + 99)(9\lambda + 1)^3 + (-59\lambda + 19)(\lambda + 9)^3}{(\lambda + 9)^3(9\lambda + 1)^3}, \\
		h''(\lambda)
		& = 1.8 \bigg[ \frac{59\lambda - 414}{(\lambda + 9)^4} + \frac{531\lambda -286}{(9\lambda + 1)^4} \bigg].
	\end{align*}
	Let $ h'(\lambda) = 0 $.
	The real solutions to the equation
	\begin{align*}
		(- 59\lambda + 99)(9\lambda + 1)^3 + (-59\lambda + 19)(\lambda + 9)^3 = 0
	\end{align*}
	are $ \lambda_1 = -0.8111 $ and $ \lambda_2 = 1 $.
	In $ [1, 2) $, the critical point is
	$\lambda^*_1 = 1$.	
	Note that
	\begin{align*}
		h''(\lambda^*_1) & = h''(1) = -0.0198 < 0,
	\end{align*}
	which implies $h(\lambda)$ has a local maximum at $1$, i.e. it is decreasing on $(1, 2)$.
	Hence, the local maximum is
	\begin{align*}
		h(\lambda^*_1) = h(1) = \frac{3}{4},
	\end{align*}
	For $ \lambda \in [2, \infty)$, we have
	\begin{align*}
		h(\lambda) 
		& = 0.075 \frac{\lambda^2 + 132}{(\lambda + 9)^2} + 0.225 \frac{243\lambda^2 + 76}{(9\lambda + 1)^2}, \\
		h'(\lambda) 
		& = 0.45 \frac{(3\lambda - 44)(9\lambda + 1)^3 + 9(27\lambda - 76)(\lambda + 9)^3}{(\lambda + 9)^3(9\lambda + 1)^3}, \\
		h''(\lambda)
		& = 1.35\bigg[\frac{- 2\lambda + 53}{(\lambda + 9)^4} + 81 \frac{- 18\lambda + 77}{(9\lambda + 1)^4}\bigg].
	\end{align*}
	Let $ h'(\lambda) = 0 $.
	The real solutions to the  equation
	\begin{align*}
		(3\lambda - 44)(9\lambda + 1)^3 + 9(27\lambda - 76)(\lambda + 9)^3 = 0
	\end{align*}
	are $ \lambda_1 = -2.2762 $ and $ \lambda_2 = 9.3689 $. 
	In $ [2, \infty) $, the critical point is $\lambda^*_2 = 9.3689$. Note that
	\begin{align*}
		h''(\lambda^*_2) & = h''(9.3689) = 0.0002 > 0.
	\end{align*}
	Then $h(\lambda)$ has a local minimum at $9.3689$, i.e. it is decreasing on $(2, 9.3689)$ and increasing on $(9.3689, \infty)$.
	Hence, the local minimum is
	\begin{align*}
		h(\lambda^*_2) = h(9.3689) = 0.7105.
	\end{align*}
	Comparing all of the local extrema and $\displaystyle \lim_{\lambda \to \infty} h(\lambda) = 3/4$, the global minimum is
	\begin{align*}
		h_{\min}(\lambda) = h(9.3689) = 0.7105,
	\end{align*}
	and the global maximum is
	\begin{align*}
		h_{\max}(\lambda) = h(1) = \frac{3}{4} = \lim_{\lambda \to \infty} h(\lambda).
	\end{align*}	
	In addition, observe that $\displaystyle \lim_{\lambda \to 2^-} h(\lambda) = h(2) $, i.e. $h(\lambda)$ is continuous at $2$. 
	Hence, $ h(\lambda) $ is decreasing on $(1, 9.3689)$ and increasing on $(9.3689, \infty)$.

\qed

\section*{Statements}
 Data sharing is not applicable to this article as no new data were created or analyzed during this study. The authors declare no conflicts of interest.

\section*{Acknowledgement}
Mingao Yuan thanks The University of Texas at El Paso for providing generous startup funds.

\end{document}